\documentclass{article}
\usepackage[utf8]{inputenc}
\usepackage[T1]{fontenc}       
\usepackage{textcomp}          
\usepackage{lmodern}         
\usepackage{xcolor}
\usepackage{geometry}
\usepackage{amsmath, amssymb, amsfonts, amsthm}
\usepackage{indentfirst}
\usepackage{cancel}
\usepackage{bm}
\usepackage{enumitem}
\usepackage{setspace}
\usepackage{hyperref}
\usepackage[all]{xy}
\usepackage{tikz}
\usepackage{diagbox}
\usepackage{array}
\usepackage{tabularx}
\usepackage{mathtools}
\usepackage{dsfont}
\usepackage{amsmath,calligra,mathrsfs}
\usepackage{adjustbox}
\usepackage{authblk}
\usepackage[textsize=tiny]{todonotes}
\usepackage[capitalize]{cleveref}
\usepackage{lipsum}
\usepackage{addfont}
\addfont{OT1}{rsfs10}{\rsfs}
\usetikzlibrary{cd}
\newtheoremstyle{case}{}{}{}{}{}{:}{ }{}
\theoremstyle{case}

\usepackage{forloop}

\usepackage{tocloft}
\newcommand{\shE}{\mathcal{E}}

\newcommand{\shk}{\mathscr{K}}

\newcommand{\shn}{\mathscr{N}}
\newcommand{\sho}{\mathcal{O}}
\newcommand{\shp}{\mathscr{P}}

\newcommand{\bz}{\mathbb{Z}}

\newcommand{\bc}{\mathbb{C}}

\newcommand{\bp}{\mathbb{P}}

\newcommand*{\sheafhom}{\mathscr{H}\kern -.5pt om}
\newcommand*{\sheafext}{\mathscr{E}\kern -.5pt xt}
\newcommand*{\sheafend}{\mathscr{E}\kern -.5pt nd}

\newcommand{\Ext}{\mathrm{Ext}}

\newcommand{\Hom}{\mathrm{Hom}}
\newcommand{\rhom}{\mathrm{RHom}}
\newcommand{\Pic}{\mathrm{Pic}}
\newcommand{\ch}{\mathrm{ch}}

\newcommand{\Ker}{\mathrm{Ker}}
\newcommand{\rk}{\mathrm{rank}}

\DeclareUnicodeCharacter{2010}{-} 
\DeclareUnicodeCharacter{2011}{-} 
\DeclareUnicodeCharacter{2013}{-} 
\DeclareUnicodeCharacter{200B}{}  
\DeclareUnicodeCharacter{0335}{}
\DeclareUnicodeCharacter{2014}{---}
\DeclareUnicodeCharacter{2212}{-}
\DeclareUnicodeCharacter{00A0}{ }

\hypersetup{
colorlinks=true,
linkcolor=blue,
linkbordercolor={0 0 1}}

\theoremstyle{plain}
\newtheorem{thm}{Theorem}[section] 
\newtheorem{lemma}[thm]{Lemma} 
\newtheorem{Corollary}[thm]{Corollary} 
\newtheorem{proposition}[thm]{Proposition} 

\theoremstyle{definition}
\newtheorem{Setting}[thm]{Setting}
\newtheorem{claim}[thm]{Claim} 

\newtheorem{remark}[thm]{Remark} 

\usepackage{listings}

\lstdefinelanguage{Julia}%
  {morekeywords={abstract,break,case,catch,const,continue,do,else,elseif,%
      end,export,false,for,function,immutable,import,importall,if,in,%
      macro,module,otherwise,quote,return,switch,true,try,type,typealias,%
      using,while},%
   sensitive=true,%
   morecomment=[l]\#,%
   morecomment=[n]{\#=}{=\#},%
   morestring=[s]{"}{"},%
   morestring=[m]{'}{'},%
}[keywords,comments,strings]%
\usepackage[doi=true,isbn=false,url=false,eprint=false,maxbibnames = 50,style=alphabetic]{biblatex}
\usepackage{xurl} 
\usepackage{csquotes}
\usepackage[american]{babel}

\author{Junyu MENG}
\date{}
\title{Hilbert squares of genus 16 K3 surfaces}
\begin{document}
\maketitle
\begin{abstract}
        We consider the geometry of a general polarized K3 surface $(S,h)$ of genus 16 and its Fourier-Mukai partner $(S',h')$. We prove that $S^{[2]}$ is isomorphic to the moduli space $M_{S'}(2,h',7)$ of stable sheaves with Mukai vector $(2,h',7)$ and describe the embeddings of the projectivization of the stable vector bundle of Mukai vector $(2,-h',8)$ over $S'$ into these two isomorphic hyper-K\"ahler fourfolds. Following the work of Frédéric Han \cite{Frederic}, we explicitly construct a 3-form $t_1\in \wedge^3 V_{10}^*$, such that $S^{[2]}$ is isomorphic to the Debarre-Voisin fourfold in $G(6,V_{10})$ associated to $t_1\in \wedge^3 V_{10}^*$. Then the reconstruction program is also explained from two perspectives, i.e. in terms of families of stable representations and families of twisted cubics.
\end{abstract}
\section{Introduction}
\subsection{Motivation}
     Deformations of polarized Hilbert schemes of points on K3 surfaces provide important families of examples of polarized hyper-K\"ahler manifolds. A natural question is to give descriptions of (general) elements in these deformation families. The Debarre-Voisin construction is a nice answer to this question in a specific case. More explicitly, let $V_{10}$ be a 10-dimensional vector space, and let $t\in \wedge^3 V_{10}^*$ be a nonzero 3-form (which is usually called a trivector), then the Debarre-Voisin 4-fold $DV(t)$ is defined as a subscheme of $G(6,V_{10})$ as follows:
     $$DV(t)=\{[V_6]\in G(6,V_{10})|  \ \mathrm{the \ restricted} \ \text{3-}\mathrm{form} \ t|_{V_6} \ \mathrm{is \ zero}\}.$$
     When $DV(t)$ is smooth, it is a hyper-K\"ahler manifold of K3$^{[2]}$-type by \cite{DebarreVoisin}. The restriction of the Pl\"ucker polarization  to $DV(t)$ is denoted by $\sho_{DV(t)}(1)$, which has Beauville-Bogomolov-Fujiki square 22 and divisibility 2. The nice thing is that these $(DV(t), \sho_{DV(t)}(1))$ give a locally complete family of smooth polarized hyper-K\"ahler manifolds of K3$^{[2]}$-type.

     Let $\mathscr{M}_{DV}$ be the projective GIT quotient $\bp(\wedge^3V_{10}^*)//SL(V_{10})$ and let $\mathscr{F}$ be the period domain for polarized hyper-K\"ahler manifolds deformation equivalent to $(DV(t), \sho_{DV(t)}(1))$. Both varieties are 20-dimensional. We can summarize the previous facts by saying that there exists a dominant rational map from $\mathscr{M}_{DV}$ to $\mathscr{F}$. Actually this map is birational as shown in \cite{Ogrady}. 
     
     A further question is to compare $\mathscr{M}_{DV}$ with the Baily-Borel compactification $\overline{\mathscr{F}}$ of $\mathscr{F}$. Both $\mathscr{M}_{DV}$ and $\overline{\mathscr{F}}$ are normal projective varieties birational to each other and we would begin by studying divisors (valuations on the rational function field) on them. A divisor on $\overline{\mathscr{F}}$ is called an HLS divisor (named after Hassett, Looijenga and Shah) if it is not a divisor on $\mathscr{M}_{DV}$ (i.e. if there does not exist a divisor on $\mathscr{M}_{DV}$ which induces the same valuation on the rational function field). It was proved in \cite{DHOV} that the Heegner divisors $\mathscr{D}_{2},\mathscr{D}_{6},\mathscr{D}_{10},\mathscr{D}_{18}$ are HLS divisors; for this they used projective models of polarized K3 surfaces of genera 2,4,6,10 constructed by Mukai. However, it was left open in \cite{DHOV} whether the Heegner divisor $\mathscr{D}_{30}$ is an HLS divisor or not.

     The Heegner divisor $\mathscr{D}_{30}$ over $\overline{\mathscr{F}}$ is actually the closure of a divisor on the open subset $\mathscr{M}\subset \overline{\mathscr{F}}$ parameterizing hyper-K\"ahler fourfolds of K3$^{[2]}$-type with a polarization of Beauville-Bogomolov-Fujiki square 22 and divisibility 2 by \cite[Theorem 3.1]{DHOV}. More precisely, a general element of $\mathscr{D}_{30}$ is given as follows: consider a general polarized K3 surface $(S,h)$ of genus 16, and let $S^{[2]}$ be the Hilbert scheme of two points on $S$, then the line bundle $2h-7\delta$ is an ample polarization of square 22 and divisibility 2 (where $h$ is the divisor over $S^{[2]}$ induced by the polarization $h$ on $S$ and $2\delta$ is the class of the exceptional divisor of the Hilbert-Chow morphism $S^{[2]}\rightarrow S^{(2)}$), which makes the polarized pair $(S^{[2]},2h-7\delta)$ an element of $\mathscr{D}_{30}\cap \mathscr{M}\subset \overline{\mathscr{F}}$. 
     
    Mukai also provided projective models for general polarized K3 surfaces of genus 16 in \cite{Mukai16}. However, the case of $\mathscr{D}_{30}$ remained open because of the lack of a deep understanding of the geometry of genus 16 polarized K3 surfaces and the complexity of the projective models in \cite{Mukai16}. Nevertheless, using Gromov-Witten techniques, Georg Oberdieck proved in \cite{oberdieck} that $\mathscr{D}_{30}$ is NOT an HLS divisor. This implies that $S^{[2]}$ arises as $DV(t)$ for some trivector $t\in \wedge^3 V_{10}^*$, where $(S,h)$ is a general polarized K3 surface of genus 16. Although the existence of the trivector is guaranteed, it is not clear how to geometrically construct $t\in \wedge^3 V_{10}^*$.

    The attempt to understand the geometry and to construct the trivector was initiated in \cite{DHOV} and \cite{Frederic}. Mukai's projective model provides a rank 2 stable vector bundle $F$ with Mukai vector $(2,h,8)$ over a general genus 16 polarized K3 surface $(S,h)$. A rank 10 vector space $V_{10}$ and a rank 6 subbundle of the trivial bundle $V_{10}$ were constructed in \cite{DHOV}, which is expected to correspond to the rank 6 universal subbundle over $G(6,V_{10})$. Frédéric Han constructed in \cite{Frederic} a trivector $t_2\in \wedge^3 V_{10}$ together with a rank 6 subbundle of the trivial bundle $V_{10}^*$ over $X=\bp_{S}(F^*)$, such that fiberwise the restriction of the 3-form $t_2$ to the rank 6 subspace is zero, inducing a map from $X$ to $DV(t_2)\subset G(6,V_{10}^*)$. It was conjectured that there should exist a canonical isomorphism $V_{10}\cong V_{10}^*$ (with the hope to get the trivector $t\in \wedge^3 V_{10}^*$ corresponding to $t_2\in \wedge^3 V_{10}$) and that $X$ can be embedded in $S^{[2]}$.

    The present article begins by proving that $X$ is naturally embedded into $S'^{[2]}$, where $S'=M_S(3,h,5)$ is the unique nontrivial Fourier-Mukai partner of $S$ (when $\Pic(S)=\bz$), and that there is a natural isomorphism $V_{10}^*\cong V_{10}'$, where $V_{10}'$ is the rank 10 vector space constructed in \cite{DHOV} for $S'$. Correspondingly we have a trivector $t_1\in \wedge^3 V_{10}^*$ corresponding to $t_2'\in \wedge^3 V_{10}'$, where $t_2'$ is obtained by applying the construction in \cite{Frederic} to $S'$. (By \cite[Theorem 1.2]{Mukai2}, the roles of $S$ and $S'$ can be interchanged and we sometimes use the parallel construction simply by adding or deleting superscript $'$ without explanation). The effectiveness of this 3-form $t_1$ is proved by using the knowledge on the $\mathbb{P}^1$-bundle $X$ to uniquely determine the trivector, together with the existence result by \cite{oberdieck}. Meanwhile, using the technique developed in \cite{Ogrady}, we confirm that the morphism in \cite[Section 8]{DHOV} gives a concrete isomorphism between $S^{[2]}$ and $DV(t_1)$.

\subsection{Statement of the result}
Let $(S,h)$ be a general polarized K3 surface of genus 16. The main point of this article is to consider the ``dual" K3 surface, i.e. the moduli space $S'=M_S(3,h,5)$ of stable sheaves over $S$ with Mukai vector $(3,h,5)$, which is the unique nontrivial Fourier-Mukai partner of $S$ when $\Pic(S)=\bz$, and the related geometry.

There are two interesting spherical vector bundles over $S$: $\sho_S$ and the rank 2 bundle $F$ with Mukai vector $(2,h,8)$ which arises in Mukai's projective model in \cite{Mukai16}. Using the Fourier-Mukai transform $\phi_\shE: D^b(S)\rightarrow D^b(S')$ associated to the normalized universal sheaf $\shE$ over $S\times S'$, we can construct  two more spherical vector bundles $G_4',E_8'$ over $S'$ with Mukai vectors $(4,h',4)$ and $(8,3h',17)$ respectively:
$$\phi_\shE(F)=G_4'[0],  \ \ \ \ \phi_\shE(\sho_S)=E_8'[-2].$$ 
The rank 4 vector bundle $G_4$ over $S$ recovers the rank 4 vector bundle (under the same name $G_4$) in \cite{Frederic}, and this new way to define it makes explicit its relation with the rank 4 vector space (generated by the 4 arrows in the 4-Kronecker quiver) arising in Mukai's projective models. 

As already mentioned, we have an embedding of $X'=\bp_{S'}(F'^*)$ into $S^{[2]}$, together with an embedding of $X'$ into the moduli space $M_{S'}(2,h',7)$ of stable sheaves of Mukai vector $(2,h',7)$ over $S'$:

\begin{thm}
    There exist natural embeddings of $X'$ into both $S^{[2]}$ and $M_{S'}(2,h',7)$.
    \begin{itemize}
        \item The image of $X'$ in $M_{S'}(2,h',7)$ is equal to the Brill-Noether locus parametrizing those $[K']\in M_{S'}(2,h',7)$ such that $\Hom(K',F')\neq 0$.
        \item The image of $X'$ in $S^{[2]}$ is equal to the degeneracy locus of the morphism $H^0(S,G_4)\otimes\sho_{S^{[2]}}=\bc^8\otimes\sho_{S^{[2]}}\rightarrow G_4^{[2]}$ over $S^{[2]}$, where $G_4^{[2]}$ is the tautological transform of the vector bundle $G_4$.
    \end{itemize}
    Moreover, the embedding of $X'$ into $S^{[2]}$ induces an identification of the restriction of the rank 6 subbundle of $V_{10}\otimes\sho_{S^{[2]}}$ to $X'$ with the rank 6 subbundle of $V_{10}'^*\otimes \sho_{X'}$, which are constructed in \cite[Section 8]{DHOV} and \cite{Frederic} respectively.\label{embeddingXintro}
\end{thm}
In fact, the two hyper-K\"ahler fourfolds $S^{[2]}$ and $M_{S'}(2,h',7)$ are isomorphic (and the isomorphism is compatible with the embeddings of $X'$ above):
\begin{thm}
    There is a natural isomorphism between $S^{[2]}$ and $M_{S'}(2,h',7)$ given as follows:
    Every stable sheaf $K'$ over $S'$ with Mukai vector $(2,h',7)$ is a quotient sheaf of $E_8'$ and fits into a short exact sequence:
    \begin{equation}
            0\rightarrow \phi_\shE(\sho_Z)\rightarrow E_8'\rightarrow K'\rightarrow0. 
        \end{equation}
    where $Z$ is a length 2 subscheme of $S$ and $\phi_\shE:D^b(S)\rightarrow D^b(S')$ is the Fourier-Mukai transform associated to the normalized universal sheaf over $S\times S'$. The assignment of $Z$ to $[K']\in M_{S'}(2,h',7)$ gives the desired isomorphism.
\end{thm}
\begin{remark}
    In parallel to the fact that the structure sheaf of every length two zero-dimensional subscheme of $S$ is a quotient of $\sho_S$, here the situation is that every stable sheaf $K'$ over $S'$ with Mukai vector $(2,h',7)$ is a quotient sheaf of $E_8'$, where $E_8'$ is a rank 8 vector bundle over $S'$ obtained by applying the Fourier-Mukai transform to $\sho_S$: $\phi_\shE(\sho_S)=E_8'[-2]$.
\end{remark}
\begin{remark}
    The existence of an isomorphism between $S^{[2]}$ and $M_{S'}(2,h',7)$ has been proven in \cite[Theorem 3.3]{KapustkaVanGeemen}. Here we present an explicit construction of the isomorphism, which will give lots of geometric information.
\end{remark}
All these results help us understand better the projective models constructed by Mukai in \cite{Mukai16}. 

Let $V_{10}^*$ denote $\Hom(G_4,F)$, and let $V_{10}'^*$ denote $\Hom(G_4',F')$. The fully faithfulness of the Fourier-Mukai transform $\phi_\shE$ induces a canonical isomorphism $V_{10}^*\cong V_{10}'$. Recall that \cite{Frederic} constructed a trivector $t_2'\in \wedge^3V_{10}'$ such that there is a natural map from $X'$ to the Debarre-Voisin variety $DV(t_2')$. Now using the isomorphism $\wedge^3V_{10}^*\cong \wedge^3V_{10}'$, we get a trivector $t_1\in \wedge^3V_{10}^*$ corresponding to $t_2'\in \wedge^3V_{10}'$. The above \cref{embeddingXintro} allows us to prove the following theorem:

    \begin{thm}
        There exists a canonical isomorphism $S^{[2]}\cong DV(t_1)\subset G(6,V_{10})$.\label{conjecture}
    \end{thm}

Then we proceed to the reconstruction problem: Mukai's projective model produces general genus 16 K3 surface with $\mathcal{T}$ being the ambient variety; but now if we fix a general polarized K3 surface of genus 16, how can we find its projective model inside the variety $\mathcal{T}$ explicitly? This question is closely related to the universal property of $\mathcal{T}$.

The universal property of $\mathcal{T}$ is twofold:  $\mathcal{T}$ is a fine moduli space of quiver representations and also birational to an irreducible component of the Hilbert scheme of twisted cubics in $\mathbb{P}^3$ (see \cite{twistedcubic}). Correspondingly we will construct families of stable quiver representations and (birational) families of twisted cubics for a given genus 16 K3 surface. An interesting feature is that the families of stable quiver representations and (birational) families of twisted cubics are not canonical and their constructions are controlled by the dual K3 surfaces.

More precisely, pick $s'\in S'=M_S(3,h,5)$, which corresponds to a stable vector bundle $E$ of Mukai vector $(3,h,5)$ on S. Let $V_8'$ denote $H^0(S',G_4')$ and we have a natural evaluation map $V_8'\rightarrow G_{4,s'}'$. We can prove the following Propositions:

\begin{proposition}
    For any $[E]\in S'$, the evaluation map $E\otimes\Hom(E,F)\rightarrow F$ is a family of stable quiver representations of dimension vector $(3,2)$ for the 4-Kronecker quiver.
\end{proposition}
\begin{proposition}
    The ample line bundle $\mathcal{O}_{X}(3H-h)$ gives an embedding of $X$ into $\mathbb{P}(V_8')$ and the projection of the embedded $X$ to $\mathbb{P}(G_{4,s'}')$ is a birational family of twisted cubics parameterized by $S$.
\end{proposition}

Finally, we finish with a section involving other interseting geometric phenomena.
\paragraph{Acknowledgements.}
I would like to thank my advisors Laurent Manivel and Thomas Dedieu for leading me to this interesting topic. I am also grateful to Pietro Beri for useful conversations.

I would like to thank Frédéric Han for his patient explanations of \cite{Frederic}. I get to understand a lot of the underlying geometry and get the direction of thinking from the discussions with him. The proof of \cref{omega 2 form} was completed with his help in using Macaulay2 and the final proof of \cref{maintheorem} was pointed out by him.
\paragraph{AI disclosure}
The first daft of this article in October, 2025 was purely produced by human, and the main extra input of the second version was pointed out by Frédéric Han after the author's Phd defense. The use of AI is constraint to assistances with editorial corrections. The author takes full responsibility for the results presented in this article.

\section{Genus 16 Polarized K3 surfaces}
     We gather in this section some known information about general genus 16 polarized K3 surfaces $(S,h)$. 
     
     By the work of Mukai \cite{Mukai16}, a general $(S,h)$ can be obtained as a locally complete intersection (with respect to vector bundles) in a 12-dimensional quiver moduli space $\mathcal{T}$, defined as the moduli space of stable representations of dimension vector $(3,2)$ for the 4-Kronecker quiver (with respect to some stability parameter). Then $\mathcal{T}$ is a smooth projective Fano variety of Picard rank 1 and index 4. It comes with a (``normalized") universal family of representations $\mathscr{E}\rightarrow\mathcal{F}\otimes V$, where $\mathscr{E},\mathcal{F}$ are globally generated vector bundles of rank 3,2 respectively with $c_1(\mathscr{E})=c_1(\mathcal{F})=:\mathcal{O}_{\mathcal{T}}(1)$ equal to the ample generator of $\Pic(\mathcal{T})$, and $V^*$ is the 4-dimensional vector space generated by the four arrows in the 4-Kronecker quiver. We consider a general global section of $\mathscr{E}^{\oplus 2}\oplus \mathcal{F}^{\oplus 2}$, whose zero locus was proved in \cite{Mukai16} to be a K3 surface $S$. The genus of the polarized pair $(S,h:=\mathcal{O}_{\mathcal{T}}(1)|_S)$ is 16. Moreover, a general genus 16 polarized K3 surface can be obtained in this way.

     Naturally, we have two interesting bundles over $S$, namely $E:=\mathscr{E}|_S,  \ F:=\mathcal{F}|_S$ with Mukai vector $(3,h,5)$ and $(2,h,8)$ respectively. These vector bundles are all slope stable, hence $F$ is the unique stable vector bundle over $S$ with given Mukai vector $(2,h,8)$. However, $E$ is not unique and the moduli space $S'$ of stable sheaves of Mukai vector $(3,h,5)$ is the unique nontrivial Fourier-Mukai partner of $S$ (which will be studied in the next section). In view of \cref{ext1=0} and \cref{stable rep}, for any $[E]\in S'$, the natural map $E\rightarrow F\otimes \Hom(E,F)^*$ is also a stable family of representations for the 4-Kronecker quiver and will induce a morphism from $S$ to $\mathcal{T}$ by the universal property of $\mathcal{T}$ (after fixing an isomorphism $\Hom(E,F)^*\cong V$). When $[E]\in S'$ varies, the four-dimensional space $\Hom(E,F)$ will also vary and they patch together to give a rank 4 vector bundle $G_4'$ over $S'$, which is a slope stable spherical bundle (\cref{StableG4}).

The possible Debarre-Voisin construction was conducted in \cite{DHOV} as follows. Let $S^{[2]}$ be the Hilbert scheme of two points on $S$. Let $q_1:\mathrm{Bl}_\Delta(S\times S)\rightarrow S^{[2]}$ be the 2:1 quotient map and let $q_2:\mathrm{Bl}_\Delta(S\times S)\rightarrow S\times S$ be the blow up morphism. The tautological transform of any vector bundle over $S$, for example $F$, is defined to be $q_{1*}(q_2^* (\sho_S\boxtimes F))=: F^{[2]}$. Then there exists a natural surjective bundle map $\mathrm{Sym}^2(F^{[2]})\rightarrow (\mathrm{Sym}^2F)^{[2]}$, whose kernel is a rank 4 vector bundle $Q_4$, with $c_1(Q_4)=2h-7\delta$ (where $2\delta$ is the class of the exceptional divisor of the Hilbert-Chow morphism $S^{[2]}\rightarrow S^{(2)}$).
$$0\rightarrow Q_4\rightarrow \mathrm{Sym}^2(F^{[2]})\rightarrow (\mathrm{Sym}^2F)^{[2]}\rightarrow 0.$$
Using the above short exact sequence, $H^0(S^{[2]},Q_4)$ is equal to the kernel of $\mathrm{Sym}^2(H^0(S,F))\rightarrow H^0(S,\mathrm{Sym}^2F)$.

Let $X$ denote the $\bp^1$-bundle $\bp_{S}(F^*)$, let $\sho_{X}(-H)$ be the relative tautological subbundle and let $W_{10}$ denote $H^0(S,F)=H^0(X,\sho_X(H))$. It is shown in \cite{Frederic} that $\sho_X(H)$ is very ample and induces an embedding of $X$ into $\bp(W_{10}^*)$. Moreover, $Q_4$ is globally generated by $H^0(S^{[2]},Q_4)$, which is 10-dimensional and denoted by $V_{10}$. In geometric terms, $V_{10}\subset \mathrm{Sym}^2W_{10}$ is a family of quadrics in $\bp(W_{10}^*)$ which cut out $X$ in $\bp(W_{10}^*)$ (\cite[Theorem 2.3.2]{Frederic}).

Viewing elements of $V_{10}$ as quadrics over $\bp(W_{10}^*)$, we can consider the linear syzygies between them: $\Ker(V_{10}\otimes W_{10}\rightarrow S^3W_{10})$, which is actually an 8-dimensional vector space that we denote by $V_8$, together with the canonical map $V_8\rightarrow V_{10}\otimes W_{10}$. It turns out that $V_8$ is the space of global sections of a rank 4 vector bundle $G_4$ constructed in \cite[Lemma 3.3.2]{Frederic} (which corresponds to the $G_4'$ above by \cref{compareG4}. Such a bundle also appeared in \cite[Proposition 7.7]{KuznetsovG4} over a genus 16 K3 surface of higher Picard rank). The rank 4 vector bundle $G_4$ is involved in the Debarre-Voisin construction. There is a short exact sequence over $X$:
\begin{equation}
    0\rightarrow N_X(-2H)\rightarrow V_{10}^*\otimes \sho_X\rightarrow G_4^*(H)\rightarrow 0,\label{N_X}
\end{equation}
where $N_X$ is the normal bundle of the embedding $X\hookrightarrow \bp(W_{10}^*)$ (we will give another explanation in (\ref{rank 6 subbundle})).

A trivector $t_2\in \wedge^3 V_{10}$ was constructed such that the restriction of the 3-form to every fiber of $N_X(-2H)\subset V_{10}^*$ is zero. First one begins with a unique element $\omega$ in $\wedge^2V_8$ (up to scalar) such that the image of $\omega$ via the map $\wedge^2V_8=\wedge^2H^0(S,G_4)\rightarrow H^0(S,\wedge^2G_4)$ is zero (we will give another possible understanding of $\omega$ in \cref{2form'}). Take the second wedge product of $V_8\rightarrow V_{10}\otimes W_{10}$ and apply Schur functor, we get $\wedge^2 V_8\rightarrow \wedge^2 V_{10}\otimes S^2W_{10}$. The image of $\omega$ via this map actually lies in $\wedge^2 V_{10}\otimes V_{10}\subset \wedge^2 V_{10}\otimes S^2W_{10}$, and gives the 3-form $t_2\in \wedge^3V_{10}$. As a result, this induces a morphism from $X$ to $DV(t_2)\subset G(6,V_{10}^*)$.

\section{Fourier-Mukai Partner $S'$}
Let $(S,h)$ be a general polarized K3 surface of genus 16 with $\Pic(S)=\bz h$. Consider the Mukai vector $v_0=(2,h,8)$, with $v_0^2=-2$. By \cite[Theorem 0.1]{YoshiokaExistence}, $M_S(v_0)$ consists of one single point, which we denote by $[F]$. The usual argument about comparing $F$ with $F^{**}$ shows that $F$ is locally free. By construction, $F$ is Gieseker stable and thus slope semistable. Since $\Pic(S)=\bz$, it follows immediately that $F$ is slope stable.

\begin{lemma}
    For a general genus 16 polarized K3 surface $(S,h)$, we have $H^1(S,F)=H^2(S,F)=0$ and $F$ is globally generated by $W_{10}:=H^0(S,F)$.\label{h^1(F)=0}
\end{lemma}
\begin{proof}
    Notice that $\chi(S,F)=10$. We only need to verify the statements for one single example of genus 16 polarized K3 surface, and this has been done in \cite{Mukai16}. 
    
    We give an alternative proof of the equality $H^1(S,F)=0$ here for future convenience, because the same proof will hold for any locally free stable sheaf with Mukai vector $(2,h,7)$. We argue by contradiction.
    If $H^1(S,F)\cong\Ext^1(F,\sho)^*\neq 0$, then there exists a nontrivial extension vector bundle $B$ which fits into the following short exact sequence:
    $$0\rightarrow \sho\rightarrow B\rightarrow F\rightarrow0.$$
    We claim that the vector bundle $B$ is slope semistable (and thus slope stable). Otherwise there exists a slope semistable subsheaf $B'\subset B$ with $\mu(B')>\mu(B)=\frac{1}{3}(h,h)$ and hence $\mu(B')\geq\frac{1}{2}(h,h)$. We can moreover assume that $B'$ is reflexive. The composition map $B'\hookrightarrow B\rightarrow F$ cannot be zero and must be injective by the slope stability of $\sho$ and $F$. Moreover, the slope stability of $F$ forces $\mathrm{rank}(B')=2$ and $c_1(B')=h$. Thus the two reflexive sheaves $B',F$ coincide outside finitely many points and they must be the same. This gives a splitting of the above short exact sequence and this contradiction shows that $B$ is slope stable.

    By calculation, $v(B)=(3,h,9)$. However the self-pairing of the Mukai vector of a slope stable sheaf should be at least $-2$. This leads to a contradiction and shows that $H^1(S,F)=0$.
\end{proof}

We also consider the Mukai vector $(3,h,5).$ By \cite[Lemma 1.2]{Yoshioka}, all the Gieseker semistable sheaves $E$ with Mukai vector $(3,h,5)$ are actually slope stable vector bundles.

\begin{lemma}
    For any $E\in M_{S}(3,h,5)$, we have $\Hom(F,E)=\Ext^1(F,E)=0$ and $\Hom(E,F)=\bc^4$.\label{ext1=0}
\end{lemma}
\begin{proof}
    Both $E$ and $F$ are slope stable bundles with $\mu(E)<\mu(F)$, thus we have $\Hom(F,E)=0$.
    By \cite[Lemma 2.1]{ToshiokaLemma}, any nonzero element in $\Ext^1(F,E)$ would induce an extension vector bundle $B$ which was proved to be slope stable. In particular, we should have $v(B)^2\geq-2$. This is absurd since $v(B)=(5,2h,13)$. The last fact follows from Serre duality and Grothendieck-Riemann-Roch formula.
\end{proof}

We let $S'$ denote the projective moduli space $M_{S}(3,h,5)$ (which is also a K3 surface), and let $\shE^*$ denote the universal sheaf on $S\times S'$. As mentioned above, $\shE^*$ is actually a vector bundle and every fiber of it is slope stable. Using $\shE$ as the kernel for the Fourier-Mukai transform, we get an equivalence of derived categories $\phi_\shE: D^b(S)\cong D^b(S')$ together with a Hodge isometry between extended Mukai lattices $\tilde{H}^*(S,\bz)\cong\tilde{H}^*(S',\bz)$. In particular, $\Pic(S)\cong \Pic(S')$, and we use $h'$ to denote the corresponding genus 16 polarization on $S'$.

\begin{lemma}
    $S$ is not isomorphic to $S'$. Moreover, $S$ is isomorphic to the moduli space of stable sheaves on $S'$ with Mukai vector $(3,h',5)$.
\end{lemma}
\begin{proof}
    By \cite[Theorem 2.1]{FMofK3}, all possible Fourier-Mukai partners of a genus 16 polarized K3 surface $S$ with $\Pic(S)=\bz.h$ could only be $M_S(1,h,15)$ and $M_S(3,h,5)$ and these two are not isomorphic. The former one is actually isomorphic to $S$ itself: For a torsion free sheaf $\mathcal{I}$ with Mukai vector $(1,h,15)$, its double dual must be equal to $\sho(h)$ and $\sho(h)/\mathcal{I}$ must be the skyscraper sheaf supported on a single point of $S$. Conversely given any point $s\in S$, the kernel of $\sho(h)\rightarrow\sho(h)|_{s}$ is a torsion free sheaf with the prescribed Mukai vector.

    We have seen that $S'=M_S(3,h,5)$ together with the polarization $h'$ is also a genus 16 polarized K3 surface with Picard rank 1 and $S$ is a nontrivial Fourier-Mukai partner of $S'$. Again by \cite[Theorem 2.1]{FMofK3}, we see that $S\cong M_{S'}(3,h',5)$.
\end{proof}

\begin{remark}
    The above argument only works for K3 surfaces with Picard rank 1. However, we can remove the Picard rank 1 assumption by appealing to \cite[Theorem 1.2]{Mukai2} instead. The statement that $S\neq S'$ is an open condition so that the above proof is enough for general $S$.
\end{remark}

Viewing $S$ as the moduli space of stable sheaves over $S'$ with Mukai vector $(3,h',5)$, we should also have a universal sheaf over $S\times S'$, which we would like to compare with the other universal sheaf $\mathcal{E}^*$. The universal sheaves are only determined up to a twist by line bundles. Mukai proved in \cite{Mukai2} that we can normalize $\mathcal{E}^*$ such that the following setting holds:
\begin{Setting}
    Viewing $S'$ as the moduli space of stable sheaves over $S$ with Mukai vector $(3,h,5)$, the universal sheaf over $S\times S'$ is $\shE^*$. Viewing $S$ as the moduli space of stable sheaves over $S'$ with Mukai vector $(3,h',5)$, the universal sheaf over $S\times S'$ is $\shE$. 
\end{Setting}
Obviously, we have:
\begin{Corollary}
   $ \det(\mathcal{E})=\sho(h'-h)$.
\end{Corollary}
With this normalized Fourier-Mukai kernel $\mathcal{E}$, we come to investigate its action on the extended Mukai lattices. We focus on the algebraic $(1,1)$ part of the extended Mukai lattice, which is isomorphic to $\bz\oplus NS(S)\oplus \bz$. 
By the setting above, we see that $\phi_{\shE}$ must map the skyscraper sheaf $k(s), \ s\in S$ to a stable vector bundle with Mukai vector $(3,h',5)$. And the inverse Fourier-Mukai transform maps $k(s'),  \ s'\in S'$ to a stable vector bundle with Mukai vector $(3,h,5)$ shifted by 2.
If we use $H\phi_\shE$ to denote the cohomological Fourier-Mukai transform corresponding to $\phi_\shE$, then we have:
$$H\phi_\shE(0,0,1)=(3,h',5), \ H\phi_\shE(3,h,5)=(0,0,1).$$
Since $H\phi_\shE$ preserves the intersection form on the lattice and maps the algebraic $(1,1)$ part to the algebraic $(1,1)$ part, we see that $H\phi_\shE$ must map $(0,h,10)$ to $\pm (0,h',10)$ (which generates the orthogonal subspace to $(3,h',5)$ and $(0,0,1)$). Since $H\phi_\shE$ must map integral elements to integral elements, we conclude that $H\phi_\shE(0,h,10)=-(0,h',10)$. (Otherwise $H\phi_\shE(2,h,8)$ would not be an integral element for example.)

Now we come to consider the image of the rank 2 stable bundle $F$ via the functor $\phi_\shE$. This image must be a rank 4 vector bundle using \cref{ext1=0}, which we denote by $G_4'$.
$$\phi_\shE(F)=G_4'[0].$$
Our next aim is to prove that $G_4'$ is a $h'$-slope stable vector bundle. To do this, we employ an argument which appeared in the proof of Proposition 3.2 in \cite{YoshiokaStable}. To prepare for the proof, we introduce the following terminology and facts as in \cite{YoshiokaStable}. 

We denote the projections from $S\times S'$ to $S$ and $S'$  by $p_1$ and $p_2$ respectively. For $\mathcal{G}^.\in D^b(S)$, we define the following contravariant functor $\psi_\shE$ as $\psi_\shE(\mathcal{G}^.)=\phi_{\shE^*}(\mathcal{G}^{.} \ \check{} \ )=p_{2*}(\mathcal{E}^*\otimes p_1^*(\mathcal{G}^{.}) \ \check{} \ )$. Using relative Serre duality, we see that its inverse is given by $\hat{\psi}_\shE(\mathcal{G}^.)=\phi_\shE^{-1}(\mathcal{G}^{.} \ \check{} \ )[-2]=p_{1*}(\mathcal{E}^*\otimes p_2^*(\mathcal{G}^{.}) \ \check{} \ )$ for $\mathcal{G}^.\in D^b(S')$. We denote $\mathcal{H}^i(\psi_\shE(\mathcal{G}^.))$ by $\psi_\shE^i(\mathcal{G}^.)$ and denote $\mathcal{H}^i(\hat{\psi}_\shE(\mathcal{G}^.))$ by $\hat{\psi}_\shE^i(\mathcal{G}^.)$. Using the spectral sequence associated to the composition of derived functors $\psi_\shE\circ\hat{\psi}_\shE=\mathrm{Id}_{D^b(S')}$, for any coherent sheaf $A$ over $S'$ we have:
\begin{equation}
    \psi_{\mathcal{E}}^p(\hat{\psi}_\shE^0(A))=0, \ p=1,2 \ \ \ \ \text{and}  \ \ \ \ \psi_{\mathcal{E}}^p(\hat{\psi}_\shE^2(A))=0, \ p=0,1.\label{spectralSeq}
\end{equation}
\begin{proposition}
    The Mukai vector of $G_4'$ is equal to $(4,h',4)$ and $G_4'$ is $h'$-slope stable. In particular, $G_4'$ is the unique stable sheaf over $S'$ with this Mukai vector. \label{StableG4}
\end{proposition}
\begin{proof}
    The computation of the Mukai vector of $G_4'=\phi_\shE(F)$ follows directly from the previous description of $H\phi_\shE$. In particular, slope stability for $G_4'$ is equivalent to slope semistability.
    Notice that $\psi_\shE(F)$ is equal to $G_4'^*[-2]$, we only need to show that $\psi_\shE^2(F)$ is slope semistable. Assume that $\psi_\shE^2(F)=G_4'^*$ is not slope semistable, then we can find a nonzero slope stable sheaf $B$ with $\mu(B)<\mu(G_4'^*)=-\frac{1}{4}(h',h')$ fitting into the following short exact sequence
    $$0\rightarrow A\rightarrow G_4'^*\rightarrow B\rightarrow 0.$$
    Applying the contravariant functor $\hat{\psi}_\shE$, we get the exact sequence
    $0\rightarrow \hat{\psi}_\shE^1(A)\rightarrow \hat{\psi}_\shE^2(B)\rightarrow F\rightarrow \hat{\psi}_\shE^2(A)\rightarrow 0$
    with $\hat{\psi}_\shE^0(B)=0$ and $\hat{\psi}_\shE^1(B)=\hat{\psi}_\shE^0(A)$.
    Since $\mathrm{rank}(B)\leq 3$, we conclude that $\mu(B)\leq-\frac{1}{3}(h',h')$. 
    
    In the case where $B$ is not isomorphic to any $\mathcal{E}^*|_{\{s\}\times S'}$ for any $s\in S$, we have $\Ext^2(B,E'^*)\cong \Hom(E'^*,B)^*$
    $=0$ for any slope stable $E'$ with Mukai vector $(3,h',5)$. Then $\hat{\psi}_\shE^2(B)=0$. Moreover $\psi_\shE^1(\hat{\psi}_\shE^1(B))=\psi_\shE^1(\hat{\psi}_\shE^0(A))=0$ by (\ref{spectralSeq}). The spectral sequence associated to the composition of derived functors $\psi_\shE\circ\hat{\psi}_\shE=\mathrm{Id}_{D^b(S')}$ gives us 
    $$E_2^{p,q}=\psi_\shE^p(\hat{\psi}_\shE^{-q}(B))\implies \left\{
    \begin{aligned}
        B,  \ &p+q=0\\
        0, \ &\text{otherwise.}
    \end{aligned}
    \right.$$
    We see $B=0$, which is absurd.

    In the case where $B$ is isomorphic to $\mathcal{E}^*|_{\{s\}\times S'}$ for some $s\in S$, then $\hat{\psi}_\shE^1(B)=0$ and $\hat{\psi}_\shE^2(B)=k(s)$ (see for example the proof of Theorem 6.18 in \cite{HuyBook}). Consequently,
$\hat{\psi}_\shE^0(A)=0$, $\hat{\psi}_\shE^1(A)=k(s)$ and $\psi_\shE^2(\hat{\psi}_\shE^1(A))\neq0$. Then we arrive at a contradiction by applying the spectral sequence associated to the composition of derived functors $\psi_\shE\circ\hat{\psi}_\shE=\mathrm{Id}_{D^b(S')}$ to $A$.
\end{proof}
There is a rank 4 vector bundle $G_4$ constructed in \cite{Frederic}, which has Mukai vector $(4,h,4)$ and was proved to be globally generated. We denote $H^0(S,G_4)$ by $V_8$, which is eight-dimensional. Using the skew-symmetric isomorphism $V_{8}\cong V_8^*$ constructed in \cite{Frederic}, the vector bundle $G_4$ fits into the following short exact sequence:
 \begin{equation}
     0\rightarrow G_4^*\rightarrow V_8\rightarrow G_4\rightarrow0.\label{G4G4dual}
 \end{equation}
Of course the same construction (of rank 4 vector bundle $G_4$) applies on $S'$ as well.
\begin{proposition}
    Our vector bundle $G_4'$ coincides with the Fourier-Mukai partner's version of the rank four vector bundle $G_4$ on $S'$.\label{compareG4}
\end{proposition}
\begin{proof}
     We only need to prove that $G_4$ is $h$-slope stable. Assume to the contrary that $G_4$ is not $h$-slope stable, then there would exist a surjective map $G_4\rightarrow B$, where $B$ is a nonzero slope stable torsion free sheaf with $\mu(B)<\mu(G_4)=\frac{1}{4}(h,h)$. Since $B$ is globally generated, we have $0\leq\mu(B)<\frac{1}{4}(h,h)$, which forces $\mu(B)=0$ and the stable sheaf $B$ would have to be equal to $\sho_S$ by \cite[Corollary 1.6.11]{HuyBook}. Thus we get a nonzero global section of $G_4^*$. However $H^0(S,G_4^*)=0$ by the above short exact sequence, which leads to a contradiction. 
\end{proof}

\begin{lemma}
    $H^1(S,G_4)=H^2(S,G_4)=0.$
\end{lemma}
\begin{proof}
    By stability, there is no nonzero map from $G_4$ to $\sho_S$, which shows $H^2(S,G_4)=0$. If $H^1(S,G_4)\neq0$, then there exists a nontrivial extension bundle $B$ of $G_4$ by $\sho$.
    By a similar argument as in \cref{h^1(F)=0}, we can show that $B$ is a slope stable vector bundle. The slope stability of $B$ implies that $-20=(5,h,5)^2=v(B)^2\geq-2$, which leads to a contradiction and proves $H^1(S,G_4)=0$.
\end{proof}
\begin{lemma}
    $H^1(S,E)=H^2(S,E)=0$, $h^0(S,E)=\chi(S,E)=8$ for any $[E]\in M_S(3,h,5)$.\label{vanishing E}
\end{lemma}
\begin{proof}
    Similar to the previous Lemma.
\end{proof}

The above lemma allows us to conclude that for a general polarized K3 surface $(S,h)$ of genus 16, the Fourier-Mukai functor $\phi_\shE$ sends $\sho_S$ to $E_8'[-2]$, where $E_8'$ is a rank 8 vector bundle over $S'$.
$$\phi_\shE(\sho_S)=E_8'[-2].$$
\begin{proposition}
    The Mukai vector of $E_8'$ is equal to $(8,3h',17)$ and $E_8'$ is $h'$-slope stable. In particular, $E_8'$ is the unique stable sheaf over $S'$ with this Mukai vector.
\end{proposition}
\begin{proof}
    This proof follows the same line as in \cref{StableG4}. Again we have $\psi_{\shE^*}(\sho_S)=E_8'[-2]$. Essentially because $\frac{3}{8}$ is just slightly larger than $\frac{1}{3}$ (just like $-\frac{1}{4}$ is slightly larger than $-\frac{1}{3}$ in \cref{StableG4}), we can play with the two contravariant functors $\psi_{\shE^*}$ and $\hat{\psi}_{\shE^*}$ and conclude.
\end{proof}

\begin{lemma}
    $\Hom(F,G_4)=\Ext^1(F,G_4)=0$, $\Hom(G_4,F)\cong V_{10}^*$.\label{V_10^*}
\end{lemma}
\begin{proof}
    Because $F$ and $G_4$ are slope stable vector bundles with $\mu(F)>\mu(G_4)$, we have $H^0(G_4\otimes F^*)=\Hom(F,G_4)=0$. Recall that by \cite{Frederic} we have the short exact sequence 
    \begin{equation}
        0\rightarrow G_4\otimes F^*\rightarrow V_{10}\otimes\sho_S\rightarrow \Omega_S(h)\rightarrow0.\label{fundamentalSES}
    \end{equation}
    It follows from the above sequence that $h^0(S,\Omega_S(h))\geq 10$, and we have $H^1(S,\Omega_S(h))=0$ by \cite[Theorem 3.2]{BottVanishing}. Then $\chi(\Omega_S(h))=10$ shows that $H^0(S,\Omega_S(h))\cong V_{10}$ and we can conclude.
\end{proof}

\section{The $\bp^1$-bundle $X'=\bp_{S'}(F'^*)$ as a prime exceptional divisor}
Let $(S,h)$ be a general genus 16 polarized K3 surface. Recall that we have already constructed the following slope stable bundles on $S$ and $S'$ respectively:
   $$G_4\cong p_{1*}(\mathcal{E}^*\otimes p_2^*F'), \ G_4'\cong p_{2*}(\mathcal{E}\otimes p_1^*F).$$
Combining with the adjunction map $p_1^*p_{1*}(\mathcal{E}^*\otimes p_2^*F')\rightarrow(\mathcal{E}^*\otimes p_2^*F') $, we have the following map over $S\times S'$: $$\mathcal{E}\rightarrow F'\otimes G_4^*, \ \ \  \mathcal{E}^*\rightarrow F\otimes G_4'^*,$$
where we omit the notation $p_1^*,p_2^*$ when there is no confusion. Now taking the dual of the second map and composing with the first map, we get $F^*\otimes G_4'\rightarrow F'\otimes G_4^*$. Thus we have $F'^*\otimes G_4'\rightarrow F\otimes G_4^*$ and a global section of $(G_4'^*\otimes F')\otimes (G_4^*\otimes F)$. Using the K\"unneth formula over $S\times S'$, we get an element in $H^0(S',G_4'^*\otimes F')\otimes H^0(S,G_4^*\otimes F)$, thus a map $V_{10}'=H^0(G_4'^*\otimes F')^*\rightarrow H^0(G_4^*\otimes F)=V_{10}^*$.

\begin{lemma}
    The map $V_{10}'=H^0(G_4'^*\otimes F')^*\rightarrow H^0(G_4^*\otimes F)=V_{10}^*$ is an isomorphism.\label{V_10 are dual to each other}
\end{lemma}
\begin{proof} 
    We consider a more general setting. Consider $\shp\in D^b(X\times Y)$, where $X,Y$ are smooth projective varieties and denote the projections from $X\times Y$ to $X$ and $Y$ by $p_X,p_Y$ respectively. Let $A_1\in D^b(X)$ and $B_2\in D^b(Y)$. Define $B_1=p_{Y*}(\shp\otimes p_X^*A_1)$ and $A_2=p_{X*}(\shp^*\otimes p_Y^*(B_2)\otimes p_Y^*\omega_Y[\dim(Y)])$ so that $A_2$ is the image of $B_2$ under the left adjoint of the Fourier-Mukai transform $\phi_\shp$. By the property of adjunction functor, we should have $\rhom(A_2,A_1)\cong\rhom(B_2,B_1)$.
    We consider the natural maps
$$B_1\rightarrow \shp\otimes A_1, \ \ \ \  \ A_2\rightarrow \shp^*\otimes B_2\otimes \omega_Y[\dim(Y)],$$
so that we can apply the same trick to get $B_1\otimes B_2^*\rightarrow A_1\otimes A_2^*\otimes \omega_Y[\dim(Y)]$. Using the K\"unneth formula and noticing $\mathrm{R}\Gamma(B_1^*\otimes B_2\otimes \omega_Y[\dim(Y)])=\rhom(B_2,B_1)^*$ by Serre duality, we can construct a map $\rhom(B_2,B_1)\rightarrow \rhom(A_2,A_1)$.
\begin{lemma}
    The above map $\rhom(B_2,B_1)\rightarrow \rhom(A_2,A_1)$ is an isomorphism.
\end{lemma}
\begin{proof}
    The isomorphism by the adjunction formula is given step by step as follows:
    $$B_2^*\otimes B_1\rightarrow B_2^*\otimes \shp\otimes A_1=(B_2\otimes\shp^*\otimes \omega_Y[\dim(Y)])^*\otimes A_1\otimes \omega_Y[\dim(Y)]\rightarrow A_2^*\otimes A_1\otimes \omega_Y[\dim(Y)].$$
    Any element of $\rhom(B_2, B_1)$ gives rise to an element in $\mathrm{R\Gamma}(A_2^*\otimes A_1\otimes \omega_Y[\dim(Y)])=\mathrm{R\Gamma}(A_2^*\otimes A_1)\otimes \mathrm{R\Gamma}(\omega_Y[\dim(Y)])$. The adjunction map is constructed by composing the above map with the trace map $\mathrm{R\Gamma}(\omega_Y[\dim(Y)])\rightarrow \bc[0]$. Now it is obvious that our map $\rhom(B_2,B_1)\rightarrow \rhom(A_2,A_1)$ given by the K\"unneth formula coincides with the one given by adjunction.
\end{proof}
 We come back to our setting and recall that $\phi_\shE(F)=G_4'$ and $\phi_\shE(G_4)=F'[-2]$. Since $\phi_\shE$ is an equivalence, we have:
    $V_{10}^*=\Hom(G_4,F)=\Hom(F'[-2],G_4')\cong \Hom(G_4',F')^*=V_{10}'.$
\end{proof}
\begin{remark}
    Similarly, we have$$\Hom(E_8,F)\cong \Hom(\sho_{S'},G_4')=V_8',$$ $$\Hom(G_4,E_8)=\Hom(E_8[-2],G_4)^*\cong \Hom(\sho_{S'},F')^*=(W_{10}')^*.$$
Notice that the map $W_{10}'^*\otimes V_8'\rightarrow V_{10}'$ in \cite{Frederic} is exactly the composition $\Hom(G_4,E_8)\otimes \Hom(E_8,F)\rightarrow \Hom(G_4,F)$.\label{interpretation}
\end{remark}
\begin{lemma}
    Let $\phi$ be any nonzero map from $E^*$ to $G_4^*$, where $E$ is a slope stable vector bundle with Mukai vector $(3,h,5)$. Then $\phi$ is generically injective and degenerates along a length two subscheme of $S$.\label{generic injective}
\end{lemma}
\begin{proof}
    If $\mathrm{ker}(\phi)$ is nonzero, then it must be a torsion free sheaf with rank either 1 or 2. Then the slope of $\mathrm{Im}(\phi)$ is bigger than $\mu(E^*)=-\frac{1}{3}(h,h)$ and smaller than $\mu(G_4^*)=-\frac{1}{4}(h,h)$, which is impossible. 
    
    Now we prove that $\mathrm{coker}(\phi)$ is torsion free. Otherwise the kernel of $G_4^*\rightarrow \mathrm{coker}(\phi)\rightarrow $
    $\mathrm{coker}(\phi)/\mathrm{Tor}(\mathrm{coker}(\phi))$ would be a reflexive sheaf $B$ such that $E^*\subsetneqq B\subset G_4^*$. Because the rank of $B$ is equal to rank of $E^*$,  $c_1(B)=bh$ for $b\geq-1$. The stability of $G_4^*$ forces $b=-1$. Since $B$ and $E^*$ are two reflexive sheaves which coincide with each other outside a finite set, we must have $B=E^*$  which violates our assumption and proves that $\mathrm{coker}(\phi)$ is torsion free.

    Because $c_1(\mathrm{coker}(\phi))=0$ and $\ch_2(\mathrm{coker}(\phi))=-2$, $\mathrm{coker}(\phi)$ must be the ideal sheaf of a length 2 subscheme.
\end{proof}

Let $X'=\mathds{P}_{S'}(F'^*)$ and let $\sho_{X'}(-H')$ denote the relative tautological subbundle. Over $S\times X'$, we have the composition map $\mathcal{E}(-H')\hookrightarrow \mathcal{E}\otimes F'^*\rightarrow G_4^*$. Fibrewise, over $S\times \{x'=(s',f'^*)\}$, this gives a nonzero map $E^*\rightarrow G_4^*$ as in the last Lemma, where the bundle $E^*$ is determined by $s'\in S'$ and the map $E^*\rightarrow G_4^*$ is determined by $f'^*$. Using the last Lemma and Eagon-Northcott complex, we get the following short exact sequence:
\begin{equation}
    0\rightarrow \mathcal{E}(-H')\rightarrow G_4^* \rightarrow \mathcal{I}_{\mathcal{C}'}(3H'-h')\rightarrow 0,\label{Totalfamily}
\end{equation}
where $\mathcal{C}'$ is the scheme-theoretical degeneracy locus of $\mathcal{E}(-H')\rightarrow G_4^*$ over $S\times X'$.
We pick $x'=(s',f'^*)\in X'$ and restrict this sequence to $S\times \{x'\}$, we get $$0\rightarrow E^*\rightarrow G_4^* \rightarrow \mathcal{I}\rightarrow 0,$$
where $\mathcal{I}$ is the ideal sheaf of two points (or a length 2 subscheme) which we denote by $\{s_1,s_2\}$. 
 Consequently we have a map $X'\rightarrow S^{[2]}$ using the ideal sheaf $\mathcal{I}_{\mathcal{C}'}$, which sends $x'=(s',f'^*)$ to $\{s_1,s_2\}$ (we assume for convenience that the length 2 subscheme is reduced).
 
 Dualizing the above sequence, we get 
 \begin{equation}
     0\rightarrow \sho\rightarrow G_4 \rightarrow E. \label{twopoints}
 \end{equation}
  As a result, there is a unique (up to scalar) nonzero section $v$ of $G_4$ (thus an element $v\in V_8$) which is mapped to the zero section of $E$. By the Eagon-Northcott complex, this element of $V_8$ vanishes exactly on $s_1,s_2$, which shows that the image of $X'$ in $S^{[2]}$ lies in the degeneracy locus of $V_8\rightarrow G_4^{[2]}$.



We take the dual of the short exact sequence (\ref{Totalfamily}), and get
$$0\rightarrow \sho(h'-3H')\rightarrow G_4 \rightarrow \mathcal{E}^*(H').$$
Using the K\"unneth formula, we see that the injective sheaf map $\sho(h'-3H')\hookrightarrow G_4$ factors as $\sho(h'-3H')\rightarrow H^0(X',\sho(3H'-h'))^*\rightarrow V_8\otimes\sho_{X'\times S}\rightarrow G_4$. When we view the intermediate map $\sho(h'-3H')\rightarrow V_8$ as a map which lives on $X'$ (rather than on $X'\times S$), this is the map which associates the line $\mathds{C}.v\subset V_8$ to each $x\in X$ as in the fiberwise description given above. In particular, the map $\sho(h'-3H')\rightarrow V_8$ exhibits $\sho(h'-3H')$ as a subbundle of the trivial bundle $V_8$ on $X'$.
\begin{lemma}
     The linear system $|\sho_{X'}(3H'-h')|$ has no base points. Moreover, $4H'-h'$ is a very ample divisor on $X'$. 
 \end{lemma}
 \begin{proof}
     The base-point-freeness comes from the fact that $\sho(h'-3H')$ is a subbundle of the trivial bundle $V_8$ on $X'$. The divisor $4H'-h'$ is very ample because $\sho_{X'}(H')$ is very ample by \cite{Frederic}.
 \end{proof}
 Recall from section 2 that $Q_4^*$ is a subbundle of $V_{10}^*\otimes \sho_{S^{[2]}}$. In view of (\ref{fundamentalSES}), we have an inclusion of vector bundle $G_4'(-H')\subset V_{10}'\otimes \sho_{X'}$.
 \begin{proposition}
     The pullback of $Q_4^*$ to $X'$ is equal to $G_4'(-H')$ under the identification of trivial bundles $V_{10}^*\otimes\sho_{X'}\cong V_{10}'\otimes \sho_{X'}$.\label{pullback bundle}
 \end{proposition}
 \begin{proof}
     We have the following commutative diagram:
     $$
\xymatrix{
G_4'(-H') \ar[r]\ar[dr]& \mathcal{E}(H-H')\ar[r]& G_4^*(H)\\
&V_{10}'\cong V_{10}^*\ar[ur]&
}
$$
\begin{claim}
    The composition map $\sho(h'-3H')\otimes G_4'(-H')\hookrightarrow V_8\otimes V_{10}^*\rightarrow W_{10}$ is the zero map over $X'$.\label{zero pairing}
\end{claim}
\begin{proof}[Proof of the Claim]
We only need to prove that the composition of the maps on $X\times X'$ $$\sho(h'-3H')\otimes G_4'(-H')\rightarrow G_4\otimes V_{10}^*\rightarrow G_4\otimes G_4^*(H)\rightarrow \sho(H)$$is zero, because $\sho(H)$ is globally generated by $W_{10}$ and the source of the map lives on $X'$ and the target of the map lives on $S$.
    However, one can trace back the definition to see that the composition map $\sho(h'-3H')\otimes G_4'(-H')\rightarrow \sho(H)$ factors as follows: $$\sho(h'-3H')\otimes G_4'(-H')\rightarrow G_4\otimes \mathcal{E}(H-H')\rightarrow \mathcal{E}^*(H')\otimes \mathcal{E}(H-H')\rightarrow \sho(H).$$
    Since we have seen $\sho(h'-3H')$ as the kernel of the map $G_4\rightarrow \mathcal{E}^*(H')$ by the dual sequence to (\ref{Totalfamily}), we know the above composition is zero.
\end{proof}

The claim tells us that given any $x'\in X'$, we can associate to it a line $\bc.v\subset V_8$ and a four-dimensional subspace $U_4=G_4'(-H')|_{x'}\subset V_{10}^*$ such that $v\otimes U_{4}$ is sent to 0 via $V_8\otimes V_{10}^*\rightarrow W_{10}$.

On the other hand, $v\in V_8$ vanishes on two points $\{s_1,s_2\}$ as a global section of $G_4$. 
We will assume $s_1\neq s_2$, since we only need to prove the proposition over an open subset of $X'$ because both bundles mentioned in the statement are subbundles of the same trivial bundle $V_{10}^*\cong V_{10}'$ over $X'$. If we consider the pullback of the map $V_8=H^0(S,G_4)\rightarrow G_4$ over the $\bp^1$-bundle $X$, and recall that the pullback of $G_4$ is equal to the image of the restriction of the map $V_8\rightarrow V_{10}\otimes \sho_{\mathds{P}(W_{10}^*)}(1)$ on $\bp(W_{10}^*)$ to $X$, then we see that the image of $v\in V_8$ vanishes on the 3-dimensional subspace generated by $\mathds{P}(F_{s_i}^*)\subset X\subset\bp(W_{10}^*)$. Consequently the image of $v\in V_8$ in $V_{10}\otimes W_{10}$ has rank at most six and thus can be expressed as $\Sigma_{i=1}^6(v_i\otimes w_i)$. The rank is actually equal to 6: the sequence (\ref{twopoints}) is a short exact sequence outside finitely many points, and after applying $\Hom(-,F)$ we immediate see that the kernel of $\Hom(G_4,F)\cong V_{10}^*\rightarrow \Hom(\sho_s,F)=W_{10}$ is exactly $\Hom(E,F)=\mathds{C}^4$. Since $V_8$ is the space of linear syzygies between $V_{10}\subset \mathrm{Sym}^2W_{10}$, we know that all these $v_i\in V_{10}\subset H^0(\bp(W_{10}^*),\sho(2))$ vanish on this 3-dimensional projective subspace.

Recall that for a point $\{s_1,s_2\}\in S^{[2]}$, the fiber of $Q_4^*\subset V_{10}^*$ over this point consists of those elements of $V_{10}^*$ which have zero pairings with all the elements in $V_{10}\subset H^0(\bp(W_{10}^*),\sho(2))$ vanishing on the 3-dimensional subspace generated by $\mathds{P}(F_{s_i}^*)$. As a result, when $x'\in X'$ is mapped to $\{s_1,s_2\}\in S^{[2]}$, the fiber of the pullback of $Q_4^*$ over $x'$ consists of those elements of $V_{10}^*$ that have zero pairing with those $v_i$ in the last paragraph.  On the other hand, we already know that elements in $U_4=G_4'(-H')|_{x'}\subset V_{10}^*$ have zero pairing with those $v_i$, which proves the Proposition.
 \end{proof}
 \begin{thm}
     $X'$ is isomorphic to the degeneracy locus of $V_8\rightarrow G_4^{[2]}$ over $S^{[2]}$.
 \end{thm}
 \begin{proof}
     The pullback of $\det(Q_4)=2h-7\delta$ is equal to $\det(G_4'^*(H'))=4H'-h'$ which is very ample. Thus the map $X'\rightarrow S^{[2]}$ does not contract any curve and is a finite surjective map to the degeneracy locus of $V_8\rightarrow G_4^{[2]}$. 
     
     The degeneracy locus of $V_8\rightarrow G_4^{[2]}$ is an effective divisor of class $h-4\delta$. By \cite{DHOV}, the nef cone and the movable cone of $S^{[2]}$ are the same and are generated by $h$ and $h-\frac{15}{4}\delta$. In particular any other hyper-K\"ahler manifold birational to $S^{[2]}$ must be isomorphic to $S^{[2]}$ itself (see \cite[Section 3.2]{DHOV}). The pseudoeffective cone of $S^{[2]}$ is dual to the movable cone and is thus generated by $\delta$ and $h-4\delta$. Consequently, the effective divisor $h-4\delta$ cannot be (nontrivially) expressed as the sum of two other effective divisors. Since the divisor $h-4\delta$ does not lie in the movable cone, we must have $h^0(S^{[2]},\sho(h-4\delta))=1$, and we will not distinguish between the divisor class $h-4\delta$ and the prime divisor representing it. Moreover, since $q_{S^{[2]}}(h-4\delta,h-4\delta)=-2$, it is a prime exceptional divisor in the sense of \cite{Boucksom}. By \cite{Druel}, there exists a divisorial contraction from $S^{[2]}$ to some normal projective variety such that the exceptional locus is exactly $h-4\delta$. Then Theorem 4.2 in \cite{KapustkaVanGeemen} shows that the divisor $h-4\delta$ is a $\bp^1$-fibration over its image (which is a K3 surface) in this contraction. In particular, the divisor $h-4\delta$ itself is a smooth variety.

     Calculation gives $(2h-7\delta)^3.(h-4\delta)=264=(4H'-h')^3$. As a result, the finite surjective map from $X'$ to $h-4\delta$ is finite and birational with smooth target. So it must be an isomorphism by Zariski's main theorem.
 \end{proof}

\begin{proposition}
    There exists an injective morphism from $X'$ to $M_{S'}(2,h',7)$. The image consists of those $[K']\in M_{S'}(2,h',7)$ such that $\Hom(K',F')\neq 0$. The points in the image can also be described as those $[K']\in M_{S'}(2,h',7)$ such that $K'$ is not locally free, or as those $[K']\in M_{S'}(2,h',7)$ such that $H^1(S',K'(-h'))\neq0$. \label{F' locally free}
\end{proposition}
\begin{proof}
    Pick $x'=(s',f'^*)$, this gives a map $F'\rightarrow F'|_{s'}\rightarrow k(s')$ over $S'$. The kernel is a stable sheaf $K'$ with Mukai vector $(2,h',7)$ fitting into the short exact sequence
    \begin{equation}
        0\rightarrow K'\rightarrow F'\rightarrow k(s')\rightarrow 0.\label{X'_ses}
    \end{equation}
    We have $(K')^{**}=F'$ and $F'/K'\cong k(s')$, which shows the injectivity of the map $X'\rightarrow M_{S'}(2,h',7)$. Since $K'\subset F'$ coincide outside a codimension 2 subset, $K'$ is not locally free.

    Given any $[K']\in M_{S'}(2,h',7)$ together with a nonzero map $K'\rightarrow F'$, this map must be injective by the stability of $K'$ and $F'$, and the cokernel is a skyscraper sheaf. Given $[K']\in M_{S'}(2,h',7)$ which is not locally free, then $(K')^{**}$ is slope stable with Mukai vector $(2,h',t)$, where $t\geq 8$. Since $v((K')^{**})^2\geq-2$, we have $t=8$ and $(K')^{**}\cong F'$. 
    
    When $[K']$ lies in the image of $X'$, $H^1(S',K'(-h'))\neq 0$ by the above short exact sequence. When $K'$ is locally free, then $K'(-h
    ')\cong K'^*$ and $H^1(S',K'(-h'))= 0$ using the argument in \cref{h^1(F)=0}.
\end{proof}

\section{Explicit isomorphism $S^{[2]}\cong M_{S'}(2,h',7)$}
    Now we would like to construct an isomorphism between $M_{S'}(2,h',7)$ and $S^{[2]}$. The idea is expressed by the following claim:
    \begin{claim}
        Given any stable sheaf $K'$ with Mukai vector $(2,h',7)$, the first cohomology sheaf of $p_{1*}(\shE\otimes p_2^*K')$ (here $p_{1*}$ is the derived pushforward) is the structure sheaf of a length 2 subscheme of $S$.
    \end{claim} Since the Mukai vector of the complex $p_{1*}(\shE^*\otimes p_2^*K')$ is equal to $(1,0,-1)$, the above statement is requiring that the ordinary pushforward $\mathcal{H}^0(p_{1*}(\shE^*\otimes p_2^*K'))$ is equal to the structure sheaf $\sho_S$. It is interesting and important to notice that $\chi(E',K')=1$, i.e. the Mukai pairing between $(2,h',7)$ and $(3,h',5)$ is $ -1$.

\begin{proof}[Proof of the Claim] First we check the above claim for $[K']$ in the image of $X'\rightarrow M_{S'}(2,h',7)$. Applying the derived functor $p_{1*}(\shE^*\otimes p_2^*(-))$ to (\ref{X'_ses}), we have to recover the exact sequence (\ref{twopoints}) that we obtained from taking dual of Eagon-Northcott complex (or \cref{generic injective}) and the cokernel is $k(s_1)\oplus k(s_2)$ (when $s_1\neq s_2$). This shows that the above claim holds for general points in the image of $X'$ inside $M_{S'}(2,h',7)$.

    Let $\shk'$ be the universal sheaf over $S'\times M_{S'}(2,h',7)$. Over $S'\times M_{S'}(3,h',5)\times M_{S'}(2,h',7)$, we can consider $\shE^*\otimes \shk'$ and its derived pushforward to $M_{S'}(3,h',5)\times M_{S'}(2,h',7)$. Fiberwise, we have $H^2(S',E'^*\otimes K')=\Hom(K',E')^*=0$ by stability. Then the first higher direct image $\shn$ must be a torsion sheaf by upper semicontinuity of $h^1$. Since the intersection of the support of $\shn$ with some fibers of the projection $M_{S'}(3,h',5)\times M_{S'}(2,h',7)\rightarrow M_{S'}(2,h',7)$ is equal to 2 points, we can assume that for an open subset $U$ of $M_{S'}(2,h',7)$, the intersection of the support of $\shn$ with the fibers of $M_{S'}(3,h',5)\times U\rightarrow U$ is zero-dimensional. Let $N$ denote the pushforward of $\shn|_U$ via the map $M_{S'}(3,h',5)\times U\rightarrow U$, we claim that the rank of $N$ is at most $2$ over $U$. In fact, the rank can never be larger than $2$ because otherwise the ordinary pushforward $\mathcal{H}^0(p_{1*}(\shE^*\otimes p_2^*K'))$ would be a torsion free sheaf with Mukai vector $(1,0,t)$, with $t>1$, which is absurd.

    The rank of $N$ can never be equal to 1 over $U\subset M_{S'}(2,h',7)$ by the following reasoning. If not, $\mathcal{H}^1(p_{1*}(\shE^*\otimes p_2^*K'))$ is a skyscraper sheaf supported on a single point $s_1\in S$ for some $[K']\in U\subset M_{S'}(2,h',7)$ and $\mathcal{H}^0(p_{1*}(\shE^*\otimes p_2^*K'))$ would be the ideal sheaf $\mathcal{I}_{s_2}$ of another point $s_2\in S$. Consider the exact triangle 
    $$\mathcal{I}_{s_2}[0] \rightarrow p_{1*}(\shE^*\otimes p_2^*K') \rightarrow k(s_1)[-1]\rightarrow \mathcal{I}_{s_2}[1].$$
    If we apply the derived functor $p_{2*}(\shE\otimes p_1^*(-))$ to the above exact triangle and take the cohomology, we get $0=\mathcal{H}^0(p_{2*}(\shE\otimes p_1^*(k(s_1)[-1])))\cong\mathcal{H}^1(p_{2*}(\shE\otimes p_1^*(\mathcal{I}_{s_2})))$.
    Using the sequence $0\rightarrow \mathcal{I}_{s_2}\rightarrow\sho_S\rightarrow\sho_{s_2}\rightarrow0$, one can show that $\mathcal{H}^1(p_{2*}(\shE\otimes p_1^*(\mathcal{I}_{s_2})))=E'$ by \cref{vanishing E}, where $E'$ corresponds to $s_2\in S=M_{S'}(3,h',5)$. This leads to a contradiction.

    Similarly, the rank of $N$ can never be equal to 0 over $U\subset M_{S'}(2,h',7)$. Otherwise we would have $p_{1*}(\shE^*\otimes p_2^*K')=I_{Z}$ where $Z$ is a length two subscheme of $S$, for some $[K']\in U\subset M_{S'}(2,h',7)$. However, the cohomology of $p_{2*}(\shE\otimes p_1^*\mathcal{I}_Z)$ is not concentrated in a single degree by a simple calculation.

    As a result, when $[K']\in U\subset M_{S'}(2,h',7)$, we must have $\mathcal{H}^0(p_{1*}(\shE^*\otimes p_2^*K'))=\sho_S$ and the length of $\mathcal{H}^1(p_{1*}(\shE^*\otimes p_2^*K'))$ is equal to 2. In order to show that $\mathcal{H}^1(p_{1*}(\shE^*\otimes p_2^*K'))$ is equal to the structure sheaf of a length 2 subscheme, we only need to exclude the possibility that $\mathcal{H}^1(p_{1*}(\shE^*\otimes p_2^*K'))=k(s)\oplus k(s)$, for some $s\in S$. If so, we consider the image of the exact triangle $\sho_S[0] \rightarrow p_{1*}(\shE^*\otimes p_2^*K') \rightarrow (k(s)\oplus k(s))[-1]\rightarrow \sho_{S}[1]$
    via the derived functor $p_{2*}(\shE\otimes p_1^*(-))$ and get
    $$0\rightarrow E'\oplus E' \rightarrow E_8'\rightarrow K'\rightarrow 0,$$
    where $E'$ is given by $s\in S=M_{S'}(3,h',5)$. Recall that the inverse Fourier-Mukai transform gives $p_{1*}(\shE^*\otimes p_2^*E_8')=\sho_S$, and this leads to $\dim(\Hom(E',E_8'))=1$. Consequently the map $E'\oplus E' \rightarrow E_8'$ cannot be injective and this leads to a contradiction.

    Up until now, we have constructed a rational map from $M_{S'}(2,h',7)$ to $S^{[2]}=\mathrm{Hilb}^2(M_{S'}(3,h',5))$, which is compatible with the embeddings of $X'$. The locus where this rational map is well defined consists of those $[K']$ such that the dimension of the support of $\mathcal{H}^1(p_{1*}(\shE^*\otimes p_2^*K'))$ is 0. And for such a $[K']$ in this definition domain, we have the following short exact sequence (and only for those $[K']$ such that the associated set $\{s_1,s_2\}$ is not a singleton):
    \begin{equation}
        0\rightarrow E'_1\oplus E'_2 \rightarrow E_8'\rightarrow K'\rightarrow 0.\label{K' ses}
    \end{equation}
    This shows that for a general $[K']\in M_{S'}(2,h',7)$, we have $\Hom(E_8',K')\neq0$. By semicontinuity, $\Hom(E_8',K')$
    $\neq 0$ for any $[K']\in M_{S'}(2,h',7)$ (this is the essential point of the whole argument presented here).

    Now we pick any $[K']\in M_{S'}(2,h',7)$ and any nonzero map $E_8'\rightarrow K'$. By a similar argument as in \cref{generic injective}, we can show that the image of $E_8'\rightarrow K'$ is a rank 2 subsheaf $K_0'\subset K'$, with first Chern class $h'$. Since $K_0'$ and $K'$ only differ at finitely many points, we can easily see that $\mathcal{H}^1(p_{1*}(\shE^*\otimes p_2^*K'))$ has zero-dimensional support if we can prove $\mathcal{H}^1(p_{1*}(\shE^*\otimes p_2^*K_0'))$ has zero-dimensional support.
    We have the short exact sequence 
    $$0\rightarrow \Ker\rightarrow E_8'\rightarrow K_0'\rightarrow0.$$
    By the stability of $E_8'$, the sheaf $\Ker$ must be slope semistable with slope $\frac{1}{3}(h',h')$. Using again that $p_{1*}(\shE^*\otimes p_2^*E_8')=\sho_S[0]$, we only need to prove that the support of $\mathcal{H}^2(p_{1*}(\shE^*\otimes p_2^*\Ker))$ is zero-dimensional. We invoke the cohomology and base change theorem and look for those $[E']\in M_{S'}(3,h',5)$ such that $\Hom(\Ker, E')\neq 0$. If there is a nonzero map $\Ker\rightarrow E'$, then by the slope stability of $E'$, this map must be surjective outside finitely many points and the kernel is a rank 3 slope semistable (and thus slope stable) sheaf with Mukai vector $(3,h',t)$, with $t\geq 5$. Because the self pairing of the Mukai vector of a slope stable sheaf should be at least $-2$, again we must have $t=5$. Consequently we know that this rank 3 slope stable sheaf must represent a point in $M_{S'}(3,h',5)=S$ and the map $\Ker\rightarrow E'$ is surjective. Then it is easy to see that the support of $\mathcal{H}^2(p_{1*}(\shE^*\otimes p_2^*\Ker))$ consists of at most two points.
\end{proof}
    \begin{thm}
        There exists an isomorphism from $M_{S'}(2,h',7)$ to $S^{[2]}=\mathrm{Hilb}^2(M_{S'}(3,h',5))$, which is compatible with the embeddings of $X'$. More explicitly, if the image of $[K']\in M_{S'}(2,h',7)$ in $S^{[2]}$ is a length two subscheme $Z\subset S$, then all these data fit into a short exact sequence: \label{maintheoremold}
        \begin{equation}
            0\rightarrow p_{2*}(\shE\otimes p_1^*(\sho_Z))\rightarrow E_8'\rightarrow K'\rightarrow0. \label{2h7_ses}
        \end{equation}
    \end{thm}
    \begin{proof}
        The regular morphism from $M_{S'}(2,h',7)$ to $S^{[2]}=\mathrm{Hilb}^2(M_{S'}(3,h',5))$ has been constructed above. We first construct a rational map from $S^{[2]}$ to $M_{S'}(2,h',7)$, which is the inverse of the previous regular morphism. Let $s_1=[E'_1],s_2=[E'_2]$ be two distinct points on $S$. They correspond to two slope stable vector bundles $E_1',E_2'$ with Mukai vector $(3,h',5)$. We have seen that $\dim(\Hom(E_i',E_8')),i=1,2$ is always equal to 1 and any nonzero map $E_i'\rightarrow E_8', \ i=1,2$  must be injective by slope stability. Thus we can naturally view $E_i'$ as subsheaves of $E_8'$. By the construction before the statement of the theorem, when the pair $E_1',E_2'$ is general, we have the subbundle inclusion $E_1'\oplus E_2'\hookrightarrow E_8'$ and the cokernel is a bundle $K'$ of Mukai vector $(2,h',7)$.
        
        We only need to prove that the vector bundle $K'$ is slope stable. For this, it is enough to show that $H^0(S',K'(-h'))=0$. We have $H^0(S',E_8'(-h'))=0$ by slope stability of $E_8'$. It follows from the following Lemma that $H^1(S',E_1'(-h'))=0$ for general $s_1\in S$. This gives the construction of the inverse rational map.

        Finally, recall again that any hyper-K\"ahler manifold birational to $S^{[2]}$, where $(S,h)$ is a general genus 16 polarized surface, must be actually isomorphic to $S^{[2]}$. Moreover, by \cite[Theorem 4.15]{debarre2020hyperkahlermanifolds}, there is no nontrivial birational map from $S^{[2]}$ to itself. This shows that our regular morphism from $M_{S'}(2,h',7)$ to $S^{[2]}$ is an isomorphism.
    \end{proof}

\begin{lemma}
$H^1(S,E(-h))=0$, for a general rank 3 slope stable vector bundle $E$ with Mukai vector $(3,h,5)$ over a general genus 16 polarized K3 surface.\label{to be refined}
\end{lemma}
\begin{proof}
    We only need to show the vanishing for one single example. For this, we come back to Mukai's original paper \cite{Mukai16}, where $E$ is taken to be $\mathcal{E}|_S$ in his notation. We tensor the short exact sequence (21) in \cite{Mukai16} with $\sho_S(-h)$ and get 
    $$0\rightarrow \sho_S(-a_1-b)\oplus\sho(-a_2-b)\rightarrow E(-h)\rightarrow\sho(-a_1-a_2)\rightarrow0.
    $$
    Notice that $\sho(a_1+b),\sho(a_2+b),\sho(a_1+a_2)$ are all nef and big line bundles, hence have zero higher cohomology. Now Serre duality and the above short exact sequence give us what we want.
\end{proof}
\begin{Corollary}
    For any $[K']\in M_{S'}(2,h',7)$, $\dim(\Hom(E_8',K'))=1$. In particular, any nonzero map $E_8'\rightarrow K'$ is surjective. \label{one dimensional E_8 to F}
\end{Corollary}
\begin{proof}
   We apply $\Hom(E_8',-)$ to the short exact sequence (\ref{2h7_ses}). Note that $\Hom(E_8',p_{2*}(\shE\otimes p_1^*(\sho_Z)))=0$ by slope stability. Moreover, $\Ext^1(E_8',E')=0$ for any $[E']\in M_{S'}(3,h',5)$ because $p_{1*}(\shE^*\otimes p_2^*E_8')=\sho_S[0]$, which shows that $\Ext^1(E_8',p_{2*}(\shE\otimes p_1^*(\sho_Z)))=0$.
\end{proof}

\begin{remark} 
    Applying $\Hom(-,F')$ to the short exact sequence (\ref{K' ses}), we get $0\rightarrow\Hom(K',F')\rightarrow V_8\rightarrow G_{s_1}\oplus G_{s_2}$, where the latter nonzero map is exactly a fiber of the map $V_8\rightarrow G_4^{[2]}$ on $S^{[2]}$. We can thus recover the fact that the image of $X'$ inside $M_{S'}(2,h',7)$ consists of those $[K']$ such that $\Hom(K',F')\neq 0$.
\end{remark}

\begin{remark}
    The existence of a non-explicit isomorphism between $S^{[2]}$ and $M_{S'}(2,h',7)$ was already proven in \cite[Theorem 3.3]{KapustkaVanGeemen}.
\end{remark}
The above theorem allows us to prove many facts about the rank 8 vector bundle $E_8'$.
\begin{Corollary}
    $H^1(S',E_8')=H^2(S',E_8')=0$, hence $\dim(H^0(S',E_8'))=\chi(E_8')=25$.
\end{Corollary}
\begin{proof}
    Recall that $\phi_\shE(\sho_S)=E_8'[-2]$, which implies that $H^1(S,E)=H^2(S,E)=0$ for any $[E]\in M_{S}(3,h,5)$. Similar vanishing of cohomology holds for any $[E']\in M_{S'}(3,h',5)$. In view of the short exact sequence (\ref{K' ses}), we only need to prove that $H^1(S',K')=H^2(S',K')=0$ holds for one (hence general) $[K']\in M_{S'}(2,h',7)$. This can be verified by taking $[K']$ to be in the image of $X'\rightarrow M_{S'}(2,h',7)$ and using the short exact sequence (\ref{X'_ses}) together with \cref{h^1(F)=0}.
\end{proof}
We also have a refinement of \cref{to be refined}:
\begin{Corollary}
    $H^1(S',E_8'(-h'))=0$. Moreover, $H^1(S',E'(-h'))=0$ for any $[E']\in M_{S'}(3,h',5)$ over a general genus 16 polarized K3 surface.\label{refined}
\end{Corollary}
\begin{proof}
    We tensor the short exact sequence (\ref{K' ses}) with $\sho_{S'}(-h')$. By \cref{to be refined}, we can pick $E_1',E_2'$ such that $H^1(E_1'(-h'))=H^1(E_2'(-h'))=0$. In order to prove that $H^1(S',E_8'(-h'))=0$, we only need to show that $H^1(S',K'(-h'))=0$ for a (general) $[K']\in M_{S'}(2,h',7)$, which follows from \cref{F' locally free}.

    Now we have shown that $H^1(S',E_8'(-h'))=0$. By stability, $H^0(S',K'(-h'))=0$. Consequently, $H^1(S', p_{2*}(\shE\otimes p_1^*(\sho_Z))(-h'))=0$ by (\ref{2h7_ses}). Again, $H^0(S',E'(-h'))=0$ by stability for any $[E']\in M_{S'}(3,h',5)$. The combination of these facts with (\ref{K' ses}) shows that $H^1(S',E'(-h'))=0$.
\end{proof}

\section{Identifications of bundles and trivectors}

We first identify the vector bundle $Q_4$. This should be the modular vector bundle appearing in Theorem 1.4 of \cite{Ogrady}. We cannot invoke that theorem directly because it was stated for a general member of the moduli space of polarized hyperK\"ahler manifolds, nevertheless we can turn to the construction in section 5 of \cite{Ogrady}, which was actually used to prove \cite[Theorem 1.4]{Ogrady}. 
Recall that for any spherical vector bundle $\mathcal{F}$ over a K3 surface, we can construct a vector bundle $\mathcal{F}[2]^{-}$ over $S^{[2]}$, which is modular. We can apply this construction to $\mathcal{F}=F$, and get a rank 4 modular vector bundle $F[2]^{-}$ with invariants as follows (by \cite[Proposition 5.8]{Ogrady}):
$$c_1(F[2]^{-})=2h-3\delta, \quad \Delta(F[2]^{-}):=c_2(\mathcal{E}nd(F[2]^{-}))=c_2(S^{[2]}).$$
This modular bundle has a different first Chern class from that of $Q_4$, but a twist of it shares the same first Chern class: $c_1(F[2]^{-}(-\delta))=2h-7\delta=c_1(Q_4)$.
\begin{proposition}
    The two vector bundles $F[2]^{-}(-\delta)$ and $Q_4$ over $S^{[2]}$ are isomorphic to each other. In particular, $\Delta(Q_4)=c_2(S^{[2]})$ and $H^i(S^{[2]},\mathcal{E}nd^0(Q_4))=0,  \forall \ i$.
\end{proposition}
\begin{proof}
  Let $p_1, p_2$ denote compositions of $q_2:\mathrm{Bl}_{\Delta}(S\times S)\rightarrow S\times S$ with the two projections. And let $q_1$ denote the 2:1 cover morphism $\mathrm{Bl}_{\Delta}(S\times S)\rightarrow S^{[2]}$. These notations will only be used in this proof.

  Recall that $Q_4$ is a subbundle of $\mathrm{Sym}^2(F^{[2]})$. We first construct a morphism from $q_1^*\mathrm{Sym}^2(F^{[2]})$ to $p_1^*F\otimes p_2^*F$. By definition, $F^{[2]}=q_{1*}p_1^*F=q_{1*}p_2^*F$, giving rise to two morphisms $f_1: q_1^*F^{[2]}\rightarrow p_1^*F$ and $f_2: q_1^*F^{[2]}\rightarrow p_2^*F$. In particular, we get the desired map $q_1^*\mathrm{Sym}^2(F^{[2]})\subset q_1^*F^{[2]}\otimes q_1^*F^{[2]}\rightarrow p_1^*F\otimes p_2^*F.$

  On the other hand, there is an injective map from $q_1^*F[2]^{-}$ to $p_1^*F\otimes p_2^*F$ by \cite[(5.2.3)]{Ogrady}. Since $q_1^*(F[2]^{-}(-\delta))=q_1^*F[2]^{-}(-e)$ is a subsheaf of $q_1^*F[2]^{-}$, where $e$ denotes the exceptional divisor of the blow up, we would like to show that the image of this subsheaf is equal to the image of $q_1^*Q_4$ via the map in the last paragraph (because the two bundles in the statement have the same first Chern class). This is obviously true away from the exceptional divisor. 
  
  Around some point on the exceptional divisor, we pick a local defining equation $z$ of the exceptional divisor, anti-invariant with respect to the involution. Let $e_1,e_2$ be two local frames for the bundle $F$. Then the local frames for $q_1^*F^{[2]}$ are simply $e_1,e_2,ze_1,ze_2$. Obviously $f_1(e_i)=p_1^*(e_i), f_2(e_i)=p_2^*(e_i)$, $f_1(ze_i)=z.p_1^*(e_i)$ and $f_2(ze_i)=-z.p_2^*(e_i)$ because of the anti-invariance of the local function $z$. A local frame of $q_1^*Q_4$ can now be written as $\{ze_1.e_2-ze_2.e_1, \ z^2(e_1.e_1)-ze_1.ze_1, \ z^2(e_2.e_2)-ze_2.ze_2, \ z^2(e_1.e_2)-ze_1.ze_2\}$ (note that $z^2$ can be viewed as a local function on $S^{[2]}$, whereas $z$ can not).

 On the other hand, the image of the map $q_1^*F[2]^{-}\rightarrow p_1^*F\otimes p_2^*F$ is spanned by $e_1\boxtimes e_2-e_2\boxtimes e_1, z(e_1\boxtimes e_1), z(e_2 \boxtimes e_2), z(e_1\boxtimes e_2+e_2\boxtimes e_1)$ (see the proof of \cite[Lemma 4.2]{DHOV}), and consequently the image of the subsheaf $q_1^*F[2]^{-}(-e)$ is spanned by $z$ times the above basis elements. A comparison of these with the result obtained from the last paragraph shows that $q_1^*F[2]^{-}(-\delta)$ coincides with $q_1^*Q_4$ and this isomorphism is compatible with the action of the involution. These allow us to conclude the main statement.
 The last statement follows from \cite[Proposition 5.4]{Ogrady}.
\end{proof}

Recall that for a general genus 16 K3 surface $(S,h)$, \cite{oberdieck} proved the existence of a 3-form $t\in \wedge^3 V_{10}^*$, such that there exists some abstract isomorphism $\psi:(S^{[2]}, 2h-7\delta)\cong(DV(t),\mathcal{O}_{DV(t)}(1))$ between polarized varieties. Let $Q$ denote the restriction of universal quotient bundle over $DV(t)\subset G(6,V_{10})$. We would also like to compare the bundle $\psi^*Q$ with $Q_4$.

\begin{proposition}
    The two vector bundles $Q_4$ and $\psi^*Q$ over $S^{[2]}$ are isomorphic. And up to a linear isomorphism of $V_{10}$, the morphism $\psi:S^{[2]}\rightarrow G(6,V_{10})$ is exactly induced by the surjective morphism $V_{10}\otimes\mathcal{O}_{S^{[2]}}\rightarrow Q_4$ over $S^{[2]}$.
\end{proposition}
\begin{proof}
The discriminant of $\psi^*Q$ is equal to $\psi^*c_2(DV(t))=c_2(S^{[2]})=\Delta(Q_4)$ by \cite[Lemma 8.1]{Ogrady}. Moreover both bundles are slope stable with respect to the polarization $2h-7\delta$ by the argument of \cite[Lemma 4.1.4]{Frederic}. 
   Now we can apply \cite[Corollary 1.7]{Ogrady} and conclude the isomorphism between the two vector bundles.

   The last statement follows from Koszul complex and a Borel-Weil-Bott type calculation, which shows that the restricted bundle morphism $V_{10}\otimes\mathcal{O}_{DV(t)}\rightarrow Q $ still induces an isomorphism between their spaces of global sections whenever $DV(t)$ is smooth.
\end{proof}

Let $\sigma\in \wedge^3V_{10}^*$ be a trivector, we use $c_{\sigma}$ to denote the induced map $\wedge^2V_{10}\rightarrow V_{10}^*$.

\begin{lemma}\label{lem:kernel}
Let $t_1,\sigma\in \wedge^3V_{10}^*$ be two nonzero 3-forms. Assume $\rk(t_1(x,-,-))\geq6$ for every nonzero $x\in V_{10}$. If $\ker c_{t_1}\subseteq\ker c_\sigma$, then $\sigma\in \mathbb{C} t_1$.
\end{lemma}
\begin{proof}
The transpose map $V_{10}\rightarrow\wedge^2V_{10}^*$ of $c_{t_1}$ is injective by assumption. Hence $c_{t_1}$ is surjective. Because of the inclusion $\ker c_{t_1}\subseteq\ker c_\sigma$, $c_\sigma$ factors through $c_{t_1}$ via a unique endomorphism of $V_{10}^*$. Denote the transpose of this endomorphism by $A: V_{10}\rightarrow V_{10}$. Then we have the equality $t_1(x,y, Az)=\sigma(x,y,z),\forall x,y,z\in V_{10}.$
Consequently,
$$t_1(x,y, Az)=\sigma(x,y,z)=-\sigma(x,z,y)=-t_1(x,z,Ay)=t_1(x,Ay,z).$$
Thus for any polynomial $P(A)$ in terms of $A$, we have $t_1(x,y,P(A)z)=t_1(x,P(A)y,z).$ 

Now we claim that there could only be one eigenvalue for $A$, otherwise there exist several generalized eigenspaces, i.e. $\ker((\lambda_sI-A)^{10})$ and each can be written as the image space of some polynomial $P_s(A)$ by the constructions in linear algebra. Notice that if $y,z \in V_{10}$ belong to two different generalized eigenspaces, let's say $y=P_s(A).y'\in \ker((\lambda_sI-A)^{10})$ for some $y'\in V_{10}$, then $t_1(x,y,z)=t_1(x,P_s(A)y',z)=t_1(x,y',P_s(A)z)=0$ since $P_s(A)z=0$. Thus for any $x\in V_{10}$, the generalized eigenspaces are orthogonal to each other. Since $t_1$ is antisymmetric, we can interchange the arguments $x,y,z$ without much trouble, thus when nonzero $x$ itself belongs to one generalized eigenspace, the rank of $t_1(x,-,-)$ can be at most the dimension of the generalized eigenspace. In the case that there exist at least two eigenvalues, the dimension of some generalized eigenspace must be at most 5, which contradicts the assumption $\rk(t_1(x,-,-))\geq6$.

Our goal is to prove that $A$ is some multiple of identity matrix. By replacing $A$ with $\lambda I-A$ (where $\lambda$ is the unique eigenvalue), we can assume that $A$ is nilpotent and the property $t_1(x,y, Az)=t_1(x,Ay,z)$ still holds.

We assume that $A^k\neq 0$, while $A^{k+1}=0$. Pick $0\neq x\in \mathrm{Im}(A^k)$. We know that $t_1(x,y,z)=0$ whenever $y$ or $z$ belong
 to $\Ker(A^k)$. Consequently,

$$\rk(t_1(x,-,-))\leq \rk(A^k)\leq [\frac{10}{k+1}],$$
where the last inequality follows from an inspection of the Jordan normal form of $A$. This will lead to a contradiction to the assumption unless $k=0$, i.e. $A=A^{0+1}=0$.
\end{proof}

\begin{thm}
    The induced morphism from the bundle surjection $V_{10}\otimes\mathcal{O}_{S^{[2]}}\rightarrow Q_4$ gives an isomorphism from $S^{[2]}$ to $DV(t_1)$.\label{maintheorem}
\end{thm}
\begin{proof}
    Up to a linear isomorphism of $V_{10}$, we may assume that the mentioned morphism is exactly $\psi$. We only need to show that the trivector $t$ given by \cite{oberdieck} is a multiple of $t_1$, so that $DV(t)=DV(t_1)$. We will invoke the last Lemma and the rank condition on $t_1$ in that Lemma is verified in general: by construction, this reduces to the corresponding statement for $t_2'$. The Peskine variety $$Y_{t_2'}^1:=\{x\in \mathbb{P}(V_{10}'^*)|\rk(t_2'(x,-,-))\leq 6\}$$ is proved to be smooth in \cite[Proposition 3.5.2]{Frederic}, thus there cannot exist any $x$ such that the rank is strictly less than $6$ and cause a singularity in the Peskine variety.

    Notice that the image of the $\mathbb{P}^1$-bundle $X'$ is contained in both $DV(t)$ and $DV(t_1)$. And it is exactly the bundle inclusion $N_X'(-2H')\hookrightarrow V_{10}'^*$ that induces the map $\psi|_{X'}:X'\rightarrow G(6,V_{10}'^*)$. In particular, for any $p'\in X'$, the restrictions of the 3-forms $t,t_1$ to the six-dimensional space $N_X'(-2H')|_{p'}$ are zero. Now for any $s'\in S'$ and $p_1'\neq p_2'\in X'$ lying over $s'$, the intersection $N_X'(-2H')|_{p_1'}\cap N_X'(-2H')|_{p_2'}$ is two-dimensional (see the proof of \cite[Lemma 3.3.2]{Frederic}). Such 2-dimensional intersections are parametrized by the bundle $\Omega_{S'}(-h')$ and the degree 2 wedge products of such 2-dimensional intersections exactly span $\ker(c_{t_2'})\cong \ker(c_{t_1})$ (see \cite[Corollary 3.3.4, Corollary 3.5.7]{Frederic}). As a result, we get the inclusion $\ker(c_{t_1})\subset \ker(c_{t})$, and the last Lemma implies that $t$ is a scalar multiple of $t_1$.
    \end{proof}

\section{Reconstruction}
Mukai's model shows that a general genus 16 K3 surface can be embedded into the fine quiver moduli space $\mathcal{T}$. The reconstruction question concerns, fix a general genus 16 K3 surface, how can we find such an embedding. A study of this question would involve the universal property of $\mathcal{T}$. 

On the one hand, $\mathcal{T}$ admits a stable family of quiver representations, satisfying some universal property. On the other hand, the variety $\mathcal{T}$ is actually birational to the irreducible component of the Hilbert scheme of $\bp^3$ containing twisted cubics (see \cite{twistedcubic}). More explicitly, we consider the map $\mathscr{E}\boxtimes\sho_{\bp(V^*)}\rightarrow \mathcal{F}\boxtimes \sho_{\bp(V^*)}(1)$ over $\mathcal{T}\times \bp(V^*)$. The degeneracy locus over a general slice $\{p\}\times \bp(V^*)$ is a twisted cubic contained in $\bp(V^*)$ (on special slices the degeneracy locus can be 2-dimensional).

In this section, we will construct families of stable quiver representations and families of twisted cubics over a fixed general genus 16 K3 surface. The constructions are controlled by its Fourier Mukai partner, i.e. we can do such a construction starting from a choice of an arbitrary point $s'=[E]\in S'$. 

Recall that $\mathcal{C}'\subset S\times X'$ is the degeneracy locus of the first nonzero map in (\ref{Totalfamily}). By \cref{generic injective}, the intersection of $\mathcal{C}'$ with any slice $S\times \{x'\}$ is the length two subscheme on $S$ corresponding to the image of $x'\in X'\hookrightarrow S^{[2]}$, which shows that $\mathcal{C}'$ is three-dimensional.
\begin{lemma}
    The fibers of the composition map $\mathcal{C}'\hookrightarrow S\times X'\rightarrow S$ are always one-dimensional.\label{bl_sS}
\end{lemma}
\begin{proof}
    Assume that the fiber over $s\in S$ is two-dimensional. Because $\mathcal{C}'$ is given by the incidence condition that $s$ should lie in the length 2 subscheme given by $x'$ for $(s,x')\in \mathcal{C}'$, and $X'\rightarrow S^{[2]}$ is an embedding, then $\{s,s_1\}\in S^{[2]}$ lies in the exceptional divisor $X'$ for all $s_1\in S$. Consequently, the exceptional divisor $X'$ contains the subscheme $\mathrm{Bl}_sS\subset S^{[2]}$. Let us use $e$ to denote the exceptional $\bp^1$ in the blow up $\mathrm{Bl}_s(S)$, then $\Pic(\mathrm{Bl}_sS)=\bz.h\oplus\bz.e$. Trivially, the divisor $h\in \Pic(S^{[2]})$ restricts to $h\in \Pic(\mathrm{Bl}_sS)$ and the divisor $\delta\in \Pic(S^{[2]})$ restricts to $e\in \Pic(\mathrm{Bl}_sS)$.

    On the other hand, the restriction of $\det(G_4^{[2]})=h-4\delta$ to $X'$ is equal to $\omega_{X'}=h'-2H'$ by the adjunction formula, and the restriction of $\det(Q_4^*)=-2h+7\delta$ is equal to $\det(G_4'(-H'))=h'-4H'$ by \cref{pullback bundle}. Consequently, if $\mathrm{Bl}_sS\subset X'$, then $2H'$ restricts to $3h-11e$. However, $3h-11e$ is not divisible by 2 and this leads to a contradiction.
\end{proof}
By the above Lemma, the degeneracy locus of the restriction of $\mathcal{E}(-H')\rightarrow G_4^*$ to $\{s\}\times X'$ is always a curve. Thus we can apply the Eagon-Northcott construction over $\{s\}\times X'$ and dualize to get:
\begin{equation}
\begin{split}
    0\rightarrow E'(-H')\rightarrow G_{4,s}^* \rightarrow \mathcal{I}_{C'}(3H'-h')\rightarrow 0,\\
    0\rightarrow \sho(h'-3H')\rightarrow G_{4,s}\otimes \sho_{X'}\rightarrow E'^*(H')\rightarrow \mathcal{E}xt^2(\sho_{C'},\sho_{X'}(h'-3H'))\rightarrow0,\label{2ses}
\end{split}
\end{equation}
where $E'$ is given by $s\in S$ and $C'$ is the intersection of $\mathcal{C}'$ with $\{s\}\times X'$. The first sequence is exactly the sequence in the proof of \cite[Proposition 2.2.2]{Frederic}.


Over the divisor $X\times S\subset M_{S}(2,h,7)\times S$, the universal family $\shk$ is given by the following short exact sequence:
\begin{equation}
   0\rightarrow\shk \rightarrow pr_2^* F\rightarrow pr_1^*\sho(H)|_{\Delta_S}\rightarrow 0, \label{universal family of K over X}
\end{equation}
where $pr_1$ and $pr_2$ are the projections from $X\times S$ to $X$ and $S$ respectively and $\Delta_S$ denotes the inverse image of the diagonal via the map $X\times S\rightarrow S\times S$, and the latter map is the composition $pr_2^* F\rightarrow pr_2^*F|_{\Delta_S}=pr_1^*F|_{\Delta_S}\rightarrow pr_1^*\sho(H)|_{\Delta_S}$.
\subsection{Stable Quiver Representations}
Our aim in this subsection is to prove that every $[E]=s'\in M_S(3,h,5)=S'$ (not only general ones) gives rise to a family of stable quiver representations i.e. $E\otimes G'_{4,s'}=E\otimes \Hom(E,F)\rightarrow F$ is a family of stable representations of the 4-Kronecker quiver parametrized by $S$. This family is obtained by restricting the morphism $\mathcal{E}^*\rightarrow F\otimes G_4'^*$ to the fiber $S\times \{s'\}$. The stability condition is equivalent to requiring that:
\begin{enumerate}[label=(\Alph*)]
    \item The bundle map $F^*\otimes G'_{4,s'}=F^*\otimes \Hom(E,F)\rightarrow E^*$ is surjective over $S$;
    \item The bundle map $\sho_X(-H)\otimes G'_{4,s'}\hookrightarrow F^*\otimes G'_{4,s'}\rightarrow E^*$ has rank at least $2$ everywhere over $X$.
\end{enumerate}

To achieve our goal, we will essentially use \cref{maintheoremold} and the following properties of the bundle $E_8$.

\begin{lemma}
    The map $V_8'\otimes F^*=\Hom(F^*,E_8^*)\otimes F^*\rightarrow E_8^*$ is surjective away from finitely many points.
\end{lemma}
\begin{proof}
    We pick two general points $[E_1],[E_2]\in S'$ and consider the corresponding short exact sequence as in (\ref{K' ses}):
    \begin{equation}
        0\rightarrow K^*\rightarrow E_8^*\rightarrow E_1^*\oplus E_2^*\rightarrow 0.\label{dual K ses}
    \end{equation}
    In our situation, $K$ is locally free and $\Ext^j(K,F)=0$ for $j=0,1,2$. Consequently, the composition $V_8'\otimes F^*\rightarrow E_8^*\rightarrow E_1^*\oplus E_2^*$ splits into the direct sum of two maps $F^*\otimes\Hom(F^*,E_i^*)\rightarrow E_i^*$.
     When $[E_i]\in S'$ is general, the map $F^*\otimes\Hom(F^*,E_i^*)\rightarrow E_i^*$ is surjective because this holds for the $[E_i]$ appearing in \cite{Mukai16} by the stability condition for representations of the 4-Kronecker quiver. Consequently we have two short exact sequences:
     $$0\rightarrow \Ker_i\rightarrow F^*\otimes\Hom(F^*,E_i^*)\rightarrow E_i^*\rightarrow 0,$$
     with $\Hom(F^*,\Ker_i)=0$. The vector bundle $\Ker_i$ must be slope stable. Otherwise, by the semistability of $F^{*\oplus 4}$, there would exist a slope semistable subbundle of $\Ker_i$ with slope equal to $-\frac{1}{2}(h,h)$, which would lead to $\Hom(F^*,\Ker_i)\neq 0$.

     The map $V_8'\otimes F^*\rightarrow E_8^*$ induces maps $\Ker_i\rightarrow K^*$ for $i=1,2$. If both induced maps are zero, then (\ref{dual K ses}) must split, which is absurd. Now we only have to observe that the generic rank of any nonzero map $\Ker_i\rightarrow K^*$ is equal to 2 by stability. Moreover, the first Chern class of the rank 2 image in $K^*$ must be $-h$. These arguments allow us to conclude.
\end{proof}
To proceed, now we need to understand the unique nontrivial 2-form $\omega'\in \wedge^2V_8'$ (up to scalar) lying in the kernel of the map $\wedge^2V_8'=\wedge^2H^0(S',G_4')\rightarrow H^0(S',\wedge^2G_4')$. By the fully faithfulness of the Fourier-Mukai transform $\phi_\shE$, we know $V_8'=\Hom(E_8,F)$, giving rise to a map $V_8'\rightarrow E_8^*\otimes F$. Taking wedge product and using Schur functor, we get a map $\wedge^2V_8'\rightarrow \wedge^2(E_8^*\otimes F)\rightarrow S^2E_8^*\otimes \wedge^2F=S^2E_8^*(h)$.
\begin{proposition}
    The 2-form $\omega'$ lies in the kernel of $\wedge^2V_8'\rightarrow H^0(S^2E_8^*(h))$. \label{omega 2 form}
\end{proposition}
\begin{proof}
    Consider the map $V_8'\rightarrow G_4'\rightarrow\mathcal{E}\otimes F$ over $S\times S'$. Taking the pushforward to $S$, we get $V_8'\rightarrow E_8^*\otimes F$. By definition, we know that $\omega'$ lies in the kernel of $\wedge^2V_8'\rightarrow\wedge^2 G_4'\rightarrow \wedge^2(\mathcal{E}\otimes F)$ and thus in the kernel of $\wedge^2V_8'\rightarrow p_{1*}(S^2\mathcal{E})(h)$.
    
    We have a comparison map $S^2E_8^*(h)=S^2(p_{1*}\shE)(h)\rightarrow p_{1*}(S^2\mathcal{E})(h)$ over an open subset of $S$. Using algorithms designed by Frédéric Han in \cite{HanCode}, one can find an example such that the map $S^2H^0(S',E')\rightarrow H^0(S',S^2E')$ is injective, which shows that the previous comparison map is an injective sheaf map.
    The statement of the proposition follows from the fact that $\omega'$ is in the kernel of $\wedge^2V_8'\rightarrow p_{1*}(S^2\mathcal{E})(h)$.
\end{proof}
  \begin{remark}
     We could probably redefine $\bc.\omega'\subset\wedge^2V_8'$ as the kernel of $\wedge^2V_8'\rightarrow H^0(S^2(E_8^*)(h))$ (notice that $\chi(S^2(E_8^*)(h))=27$). This is in parallel with the fact that the symplectic form on $V_8'$ is the kernel of the composition of $\wedge^2V_8'\rightarrow H^0(\wedge^2(\sho_{S'}^*\otimes G_4'))\rightarrow H^0(S^2(\sho_{S'}^*)\otimes \wedge^2G_4')$. \label{2form'}
 \end{remark}
 We have the map $E_8\rightarrow V_8'^*\otimes F$ and we can dualize it to get $V_8'\otimes F^*\rightarrow E_8^*$. Tensoring the latter map with $\sho_S(h)=\det(F)$ and using the isomorphism $V_8'\cong V_8'^*$ given by $\omega'$ (see \cite[Proposition 3.4.2]{Frederic}), we get a map $V_8'^*\otimes F \rightarrow E_8^*(h)$.
 \begin{proposition}
     The following sequence is exact:
      \begin{equation}
 0\rightarrow E_8\rightarrow V_8'^*\otimes F \rightarrow E_8^*(h)\rightarrow0.\label{E8 ses}   
 \end{equation}
 \end{proposition}
\begin{proof}
Notice that $\mathrm{ch}(V_8'^*\otimes F)=\mathrm{ch}(E_8)+\mathrm{ch}(E_8^*(h))$. The composition of $E_8\rightarrow V_8'^*\otimes F \rightarrow E_8^*(h)$ is zero by \cref{omega 2 form}. We have seen that the latter map is surjective away from finitely many points. Thus $E_8\rightarrow V_8'^*\otimes F$ is injective and its image coincides with the kernel of $V_8'^*\otimes F \rightarrow E_8^*(h)$ except over finitely many points. Using the fact that both the kernel of $V_8'^*\otimes F \rightarrow E_8^*(h)$ and $E_8$ are reflexive, we can conclude that the above sequence is exact.
\end{proof}
\begin{lemma}
    For every $[E]\in M_S(3,h,5)$, the bundle map $F^*\otimes \Hom(F^*,E^*)\rightarrow E^*$ over $S$ is surjective.\label{stability condition (A)}
\end{lemma}
\begin{proof}
    We pick $[E_2]\in M_S(3,h,5)$ such that $E_1:=E$ and $E_2$ fit into (\ref{dual K ses}), with $K$ locally free and $\Ext^j(K,F)=0$ for $j=0,1,2$ (a choice of such $E_2$ is possible by the proof of \cref{bl_sS}). Again, the composition $V_8'\otimes F^*\rightarrow E_8^*\rightarrow E_1^*\oplus E_2^*$ must split into the direct sum of two maps $F^*\otimes\Hom(F^*,E_i^*)\rightarrow E_i^*$. By the last Proposition, the map $V_8'\otimes F^*\rightarrow E_8^*$ is surjective, which immediately shows what we want.
\end{proof}

\begin{proposition}
    For any $[E]\in M_S(3,h,5)$, the evaluation map $F^*\otimes\Hom(F^*,E^*)\rightarrow E^*$ is a family of stable quiver representations of dimension vector $(2,3)$ for the 4-Kronecker quiver. Consequently, there exists a morphism (unique when we fix an isomorphism $\Hom(E,F)^*\cong V$) from $S$ to $\mathcal{T}$ such that $E\rightarrow F\otimes \Hom(E,F)^*$ is equal to the pullback to $S$ of the bundle map $\mathscr{E}\rightarrow \mathcal{F}\otimes V$ over $\mathcal{T}$.  \label{stable rep}
\end{proposition}
\begin{proof}
    We have proved (A) in \cref{stability condition (A)}. In order to prove (B), we globalize it and show that the map $\shE^*(-H)\hookrightarrow\shE^*\otimes F^* \rightarrow G_4'^*$ has rank at least 2 everywhere over $S'\times X$. We pick $(s,f^*)=x\in X$ and restrict the Fourier-Mukai counterpart of (\ref{Totalfamily}) to $S'\times \{x\}$ to get the following exact sequence: 
$$0\rightarrow E'^*\rightarrow G_4'^* \rightarrow \mathcal{I}'\rightarrow 0.$$
(Note that the fiber of $E'^*$ over $s'$ is equal to the fiber of $\mathcal{E}^*$ over $(s,s')$, hence equal to the fiber of $E$ over $s$.) Now (B) follows from the fact that the fiberwise rank of $\mathcal{I}'$ is at most 2 everywhere for the ideal sheaf $\mathcal{I}'$ of any length two subscheme of $S'$.
\end{proof}

There is a natural map from $\mathrm{Sym}^2V^*\otimes \sho_{\mathcal{T}}$ to $\mathscr{E}$ as follows: Taking the second wedge product of $\mathscr{E}\rightarrow \mathcal{F}\otimes V$ and applying the Schur functor, we get a map $\wedge^2\mathscr{E}\rightarrow \wedge^2\mathcal{F}\otimes \mathrm{Sym}^2V$. Since $\wedge^2\mathcal{F}\cong \wedge^3\mathscr{E}$, we immediately obtain the desired map $\mathrm{Sym}^2V^*\otimes \sho_{\mathcal{T}}\rightarrow\mathscr{E}$, which was known to be surjective by \cite{twistedcubic}.  
\begin{Corollary}
    Every $[E]\in M_S(3,h,5)$ is globally generated by $\mathrm{Sym}^2G'_{4,s'}$. Consequently, there exists a surjective bundle map $(\mathrm{Sym}^2G_4')(-h')\rightarrow \shE^*$ over $S\times S'$. 
\end{Corollary}
\begin{proof}
    The map $(\mathrm{Sym}^2G_4')(-h')\rightarrow \shE^*$ is obtained by applying the second wedge product and the Schur functor to $\shE^*\rightarrow F\otimes G_4'^*$ over $ S\times S'$ as above.
\end{proof}
\subsection{Twisted Cubics}

Pick a point $s'\in S'=M_{S}(3,h,5)$, which corresponds to a rank 3 vector bundle $E$ over $S$. By definition, we know that $\Hom(E,F)=G'_{4,s'}$, where we use $G'_{4,s'}$ to denote the fiber of $G_4'$ over $s'$. We have $V^*:= G'_{4,s'}$ as in Mukai's paper \cite{Mukai16}. We have shown in the last subsection that the map $E\rightarrow F\otimes V$ is a family of stable representations of the 4-Kronecker quiver with dimension vector $(3,2)$. Then by the universal property of the 12-dimensional moduli space $\mathcal{T}$ of stable representations of the 4-Kronecker quiver with dimension vector $(3,2)$, we have a map from $S$ to $\mathcal{T}$ which exhibits $S$ as (birationally) parameterizing a two-dimensional family of twisted cubics in $\mathds{P}(V^*)=\mathds{P}(G'_{4,s'})$.

Recall that $V_8'$ is the space of global sections of $G_4'$, thus we have the surjective evaluation map $V_8'\rightarrow G_{4,s'}'$ for every $s'\in S'$, with the kernel being $G_{4,s'}'^*$ (see \cite[Proposition 3.4.2]{Frederic}). The nice geometric picture that we will prove is that there is a family of twisted cubics inside $\mathbb{P}(V_8')$, given by an embedding $X\hookrightarrow \mathbb{P}(V_8')$ (the twisted cubics are the images of fibers of the $\bp^1$-bundle map $X\rightarrow S$), such that the projection of this family of twisted cubics via the rational map $\mathbb{P}(V_8')\dashrightarrow\mathbb{P}(G_{4,s'}')$ is what we desire in the last paragraph for any $s'\in S'$.

Consider over $X$ the composition $E_8\rightarrow (V_8')^*\otimes F\rightarrow (V_8')^*\otimes \sho_{X}(H)$. We obtain a map $V_8'\rightarrow E_8^*(H)$. Fibrewise over $x=(s,f^*)\in X$, this map coincides with the map $V_8'=H^0(G_4')\rightarrow H^0(E')$ (see Section 4), where $E'$ is the rank 3 Mukai bundle over $S'$ determined by $x\in X$. (Recall that $E_8^*$ is by definition the bundle whose fiber over $s$ is $H^0(E')$). We have seen from the Eagon-Northcott complex that the kernel of $V_8'=H^0(G_4')\rightarrow H^0(E')$ should be the one-dimensional space $\bc.v'\subset V_8'$ associated to $x\in X$. As a result, the bundle map $V_8'\rightarrow E_8^*(H)$ is of constant rank 7, and we get the exact sequence:
 \begin{equation}
     0\rightarrow \sho(h-3H)\rightarrow V_8'\rightarrow E_8^*(H)\rightarrow \sho(5H-2h)\rightarrow 0.\label{h-3H to V8' four terms}
 \end{equation}
 
We investigate the subbundle map $\sho(h-3H)\rightarrow V_8'$ by giving it another interpretation. We consider the composition $\Hom(E_8,K)\otimes \Hom(K,F)\rightarrow \Hom(E_8,F)=V_8'$, where $[K]\in X\subset M_{S}(2,h,7)$. Recall that the vector space $\Hom(E_8,K)$ is always one-dimensional by \cref{one dimensional E_8 to F}. Tensoring the short exact sequence (\ref{universal family of K over X}) with $pr_2^*E_8^*$ and applying the pushforward $pr_{1*}$, we recover the map $V_8'\rightarrow E_8^*(H)$ which appeared in (\ref{h-3H to V8' four terms}). This shows that the family of one-dimensional $\Hom(E_8,K)$ is parametrized by $\sho(h-3H)$.

On the other hand, applying $\mathcal{H}om(-,pr_2^*F)$ to (\ref{universal family of K over X}) and taking the pushforward $pr_{1*}$, we see that the space $\Hom(K,F)$ is always one-dimensional and the family of these one-dimensional $\Hom(K,F)$ is parametrized by $\sho_X$.

Thus the map $\sho(h-3H)\rightarrow V_8'\otimes \sho_X$ parametrizes the map $\Hom(E_8,K)\otimes \Hom(K,F)\rightarrow \Hom(E_8,F)=V_8'$.
Using this interpretation, we can show that the morphism $X\rightarrow\bp(V_8')$ induced by $\sho(h-3H)\hookrightarrow V_8'$ is an embedding.
\begin{proposition}
   The morphism $X\rightarrow\bp(V_8')$ is an embedding. The image consists of those elements in $\Hom(E_8,F)$ such that the corresponding map $E_8\rightarrow F$ is not surjective.\label{3H-h embedding}
\end{proposition}
\begin{proof}
    We begin by noticing that the image of any nonzero map $E_8\rightarrow F$ always has rank 2 and first Chern class $h$ (because of slope stability). If the image is not equal to $F$, then the image must be contained in some subsheaf $K$, where $[K]\in X\subset M_S(2,h,7)$. However, any nonzero map $E_8\rightarrow K$ must be surjective (\cref{one dimensional E_8 to F}). Consequently, the image is either $F$ or some $[K]\in M_S(2,h,7)$. 

    Clearly the image of $X$ lies in the locus $\tilde{X}$ where the corresponding map $E_8\rightarrow F$ is not surjective. When $E_8\rightarrow F$ is not surjective, the cokernel is a length one zero-dimensional torsion sheaf. This gives an inverse map from $\tilde{X}$ to the corresponding Quot scheme $X$ for the rank 2 bundle $F$.
\end{proof}

Now for $s'=[E]\in S'$, the composition map of $X\hookrightarrow \mathbb{P}(V_8')\dashrightarrow \mathbb{P}(G_{4,s'}')$ can be explained in the following way: for any $[E]\in S'$, the space $\Hom(E,E_8)$ is one-dimensional (because $\phi_{\shE^*}(\mathcal{O}_{S'})=E_8[-2]$, which implies that $p_{2*}(\mathcal{E}\otimes p_{1}^*E_8)=\mathcal{O}_{S'}[0]$), thus any element in $\Hom(E_8,K)$ induces an element in $\Hom(E,K)$ up to scalar. For $[K]\in X\subset M_S(2,h,7)$, the composition of the induced element in $\Hom(E,K)$ with the unique map from $K$ to $F$ gives an element in $\mathbb{P}(\Hom(E,F))=\mathbb{P}(G_{4,s'}')$.

Actually, we can directly consider the map $\Hom(E,K)\otimes \Hom(K,F)\rightarrow \Hom(E,F)=V^*$, which gives us a rational map $X\dashrightarrow\bp(V^*)$.

We tensor the sequence (\ref{universal family of K over X}) with $pr_2^*E^*$:
$$0\rightarrow\shk \otimes pr_2^*E^*\rightarrow pr_2^* F\otimes pr_2^*E^*\rightarrow pr_2^*E^*|_{\Delta_S}\otimes pr_1^*\sho(H)\rightarrow 0.$$
We consider the pushforward $pr_{1*}$ of the above sequence and compare with (\ref{2ses}). We get the Fourier-Mukai partner's version of (\ref{2ses}):
$$0\rightarrow \sho(h-3H)\rightarrow G'_{4,s'}\rightarrow E^*(H)\rightarrow \mathcal{E}xt^2(\sho_{C},\sho_{X}(h-3H))\rightarrow0.$$ This exact sequence shows that for a fixed $[E]\in M_S(3,h,5)$, the space $\Hom(E,K)$ is one-dimensional for a general $[K]\in X\subset M_S(2,h,7)$ (and this holds outside $C\subset X$ which is of codimension 2) and the family of these one-dimensional $\Hom(E,K)$ is parametrized by the line bundle $\sho(h-3H)$.
The map $\Hom(E,K)\otimes \Hom(K,F)\rightarrow \Hom(E,F)=V^*$, which gives us a rational map $X\dashrightarrow\bp(V^*)$ defined outside $C$, and it is exactly the map $\sho(h-3H)\rightarrow G'_{4,s'}=V^*$ that we have seen above.

Now we explain the family of twisted cubics in $\mathds{P}(G'_{4,s'})=\bp(V^*)=\bp(\Hom(E,F))$ when we fix some $s'=[E]\in S'$. Based on the construction, this should be the locus where the corresponding map $E\rightarrow F$ is not of full rank. Here we consider $\mathscr{C}_S\subset X\times \mathds{P}(\Hom(E,F))$, which is defined as follows:
$$\mathscr{C}_S=\{(s,f^*)=x\in X, \ \psi\in \bp(\Hom(E,F))| 
\ (\mathrm{rank}(\psi_s)\leq 1,) \ \mathrm{Im}(\psi_s)\subset \mathrm{ker}(f^*: F_s\rightarrow\bc)\}.$$
Obviously $\mathscr{C}_S$ is birational to the 2-dimensional family of twisted cubics (when $[E]=s'\in S'$ is general) via the natural composition map $\mathscr{C}_S\hookrightarrow X\times \mathds{P}(\Hom(E,F))\rightarrow S\times \mathds{P}(\Hom(E,F))$,
whose image is exactly 
$$\{s\in S, \ \psi\in \bp(\Hom(E,F))| 
\ \mathrm{rank}(\psi_s)\leq 1\}.$$ Moreover, $(x,\psi)\in X\times \bp(\Hom(E,F))$ belongs to $\mathscr{C}_S$ iff the map $\psi: E\rightarrow F$ factors through the sheaf $K$ corresponding to $x\in X\subset M_{S}(2,h,7)$.

\begin{proposition}
    The variety $\mathscr{C}_S$ is isomorphic to the blow up of $X$ along the curve $C$.
\end{proposition}
\begin{proof}
    Let $x=[K]\in X$ be a point outside $C$ such that the image of $\Hom(E,K)\otimes \Hom(K,F)\rightarrow \Hom(E,F)$ is not zero, and pick $\psi\neq0$ in that image, then $(x,[\psi])$ is a point in $\mathscr{C}_S$. If there exists a distinct point $(x,[\psi'])\in \mathscr{C}_S$, then $\psi'$ has to factor through $K\subset F$ and must coincide with $\psi$ up to a scalar since we have shown that $\dim(\Hom(E,K))=1$. This allows us to conclude that $\mathscr{C}_S$ is the graph of the rational map $X\dashrightarrow \bp(\Hom(E,F))$.

    On the other hand, recall that we have a surjective map $G_{4,s'}'^*\otimes \sho_X\twoheadrightarrow \mathcal{I}_C(3H-h)$ in (\ref{2ses}) whose dual map is exactly $\mathcal{O}(h-3H)\rightarrow G_{4,s'}'\otimes\sho_X$. Taking projectivization of the symmetric algebras of $G_{4,s'}'^*\otimes \sho_X\twoheadrightarrow \mathcal{I}_C(3H-h)$ gives us an embedding of $Bl_CX$ into $X\times \mathds{P}(\Hom(E,F))$ which also realizes $Bl_CX$ as the graph of the rational map.
\end{proof}

This allows us to conclude that the projection of $X$ to $\mathbb{P}(\Hom(E,F))$ is exactly birational to the family of twisted cubics appearing in Mukai's projective model since the composition of $$X\dashrightarrow Bl_CX=\mathscr{C}_S\hookrightarrow X\times \bp(\Hom(E,F))\rightarrow S\times \bp(\Hom(E,F))$$ is exactly induced by the sheaf injection $\mathcal{O}_X(h-3H)\rightarrow G_{4,s'}'=\Hom(E,F)$.

We continue to solve a delicate problem: the variety $\mathcal{T}$ is only birational to an irreducible component of the Hilbert scheme of twisted cubic, so a priori it could happen that the subvariety $S\subset \mathcal{T}$ intersects the locus where the birational map is not defined, or in other words, for some $s\in S$, the line $\bp(F^*_s)=\bp^1\subset X$ may lie in the locus $C$ where the map $X\rightarrow \bp(\Hom(E,F))$ is not defined. Finally, it turns out that this could never happen:
\begin{lemma}
    The curve $C$ can not contain any $\bp(F^*_s)=\bp^1\subset X$, where $s$ is some point of $S$.
\end{lemma}
\begin{proof}
Assume to the contrary that $C$ contains some $\bp(F^*_s)$. This would imply that the fiber of $\mathcal{C}\hookrightarrow S'\times X\rightarrow S'$ at the point $s'=[E]\in S'$ will contain $\bp(F^*_s)$. By the calculation in \cref{bl_sS}, the intersection numbers of $\bp(F^*_s)$ with $h'$ and $\delta'$ in $S'^{[2]}$ are equal to 30 and 8 respectively. 

On the other hand $\bp(F^*_s)$ would be a smooth curve in $Bl_{s'}S'\subset S'^{[2]}$ by a similar argument in \cref{bl_sS}, and the above intersection numbers allow us to conclude that the class of $\bp(F^*_s)$ is equal to $h'-8e'$, where $e'$ is the class of exceptional divisor of the blowup. The adjunction formula on the surface would then give $-2=2g(\bp(F^*_s))-2=(h'-8e')(h'-8e'+e')=-26$, which is absurd.
\end{proof}

We finish this subsection with another geometric observation.
\begin{lemma}
    For any $[E]\in M_S(3,h,5)$, the kernel of any nonzero map $E\rightarrow F$ is equal to $\sho_S$ and this map degenerates at four points of $S$ (counting multiplicity). \label{4 degenerate point}
\end{lemma}
\begin{proof}
    By the slope stability of $E,F$, we know that a nonzero map must be generically surjective with $c_1(\mathrm{Im}(E\rightarrow F))=h$. The kernel is then a reflexive sheaf of rank one with trivial first Chern class, which shows that the kernel is exactly $\sho_S$. A computation shows that the Mukai vector of the cokernel is equal to $(0,0,4)$.
\end{proof}

\begin{proposition}
    The rational map $X\dashrightarrow\bp(\Hom(E,F))$ is dominant and is generically 4-to-1.
\end{proposition}
\begin{proof}
    By \cref{4 degenerate point}, any nonzero map $E\rightarrow F$ degenerates at 4 points. Let $s\in S$ be in the degeneracy locus and we can pick $f^*\in F^*_s=\Hom_{\mathbb{C}}(F_s,\mathbb{C})$ such that the composition $E\rightarrow F\rightarrow k(s)$ is zero. Thus $E\rightarrow F$ factors through the kernel $K$ of $f^*:F\rightarrow k(s)$ and this shows that the rational map is dominant. It is straightforward that the map is generically 4-to-1.
\end{proof}

 \section{More Projective Geometry}

 \subsection{The rank-six locus in \(\bp(V_8')\)}

 We investigate another composition $\Hom(G_4,K)\otimes \Hom(K,F)\rightarrow\Hom(G_4,F)=V_{10}^*$. This recovers some key constructions in \cite{Frederic}.
 
 \begin{lemma}
    $\dim(\Hom(G_4,K))=6$, $\Ext^1(G_4,K)=\Ext^2(G_4,K)=0$ for any $[K]\in M_S(2,h,7)$.
\end{lemma}
 \begin{proof}
     Recall that $\dim(\Hom(G_4,E_8))=10, \ \Ext^1(G_4,E_8)=\Ext^2(G_4,E_8)=0$ using the fully faithfulness of Fourier-Mukai transform. Moreover $\dim(\Hom(G_4,E))=2$ and $\Ext^1(G_4,E)=\Ext^2(G_4,E)=0$ because $p_{2*}(\shE\otimes p_1^*G_4)=F'[-2]$. Now by applying $\Hom(G_4,-)$ to the Fourier-Mukai partner's version of the short exact sequence (\ref{2h7_ses}), we obtain what we want.
 \end{proof}
 We tensor the short exact sequence (\ref{universal family of K over X}) with $pr^*_2G_4^*$ and do the pushforward $pr_{1*}$ to get the following short exact sequence:
 \begin{equation}
     0\rightarrow pr_{1*}(\shk \otimes pr_2^*G_4^*)\rightarrow V_{10}^*\otimes\sho_X\rightarrow G_4^*(H) \rightarrow 0.\label{rank 6 subbundle}
 \end{equation}
 By comparing this short exact sequence with (\ref{N_X}), we see that $pr_{1*}(\shk \otimes pr_2^*G_4^*)\cong N_X(-2H)$, where $N_X$ is the normal bundle of the embedding $X\hookrightarrow \bp(W_{10}^*)$. Recall that it is the subbundle inclusion $N_X(-2H)\hookrightarrow V_{10}^*$ that defines the map $X\rightarrow DV(t_2)\subset G(6,V_{10}^*)$ in \cite{Frederic}. Moreover, the map $\bp_X(N_X(-2H))=\bp_X(pr_{1*}(\shk \otimes pr_2^*G_4^*))\rightarrow \bp(V_{10}^*)$ is exactly the restriction of the map $\mathrm{Bl}_X(\bp(W_{10}^*))\rightarrow \bp(V_{10}^*)$ to the exceptional divisor of the blow up $\mathrm{Bl}_X(\bp(W_{10}^*))$.
 \begin{Corollary}
     The rank of the composition of the bundle maps $W_{10}'^*\otimes \sho_{\bp(V_8')}(-1)\hookrightarrow W_{10}'^*\otimes V_8'\otimes \sho_{\bp(V_8')}\rightarrow V_{10}^*\otimes \sho_{\bp(V_8')}$ is exactly 6 over $X\hookrightarrow \bp(V_8')$.
 \end{Corollary}
\begin{proof}
    For $x\in X$, the corresponding point in $\bp(V_8')=\bp(\Hom(E_8,F))$ is the nonzero composition map $f_x:E_8\rightarrow K\rightarrow F$. The fiber of the above bundle map $\Hom(G_4,E_8)\otimes\sho_{\bp(V_8')}(-1)= W_{10}'^*\otimes \sho_{\bp(V_8')}(-1)\rightarrow V_{10}^*\otimes \sho_{\bp(V_8')}=\Hom(G_4,F)\otimes \sho_{\bp(V_8')}$ over the point $x\in X$ is given by composing with $f_x:E_8\rightarrow K\rightarrow F$. Since $\Hom(G_4, E_8)\rightarrow \Hom(G_4,K)$ is surjective by (\ref{2h7_ses}) and the map $\Hom(G_4,K)\hookrightarrow \Hom(G_4,F)$ is injective, we see that the rank of the map over $X$ is exactly 6.
\end{proof}
\subsection{Projective embeddings associated with \(E_8\)}
By definition, we have $p_{1*}(\mathcal{E})=E_8^*$, hence a map $p_1^*E_8^*\rightarrow \mathcal{E}$ and similarly $p_2^*E_8'^*\rightarrow\mathcal{E}^*$. We have a commutative diagram (over $S\times S'$):
\[
\begin{tikzcd}
V_8\otimes (F')^*\ar[r,two heads]\ar[d,two heads]&(E_8')^*\ar[r]\ar[d,two heads]&(G_4')^*\otimes W_{10}\ar[d, two heads]\\
G_4\otimes (F')^* \ar[r,two heads]\ar[d,hook]& \mathcal{E}^*\ar[r,hook]\ar[d,hook]& (G_4')^*\otimes F\ar[d,hook]\\
G_4\otimes (W_{10}')^*\ar[r]&E_8\ar[r,hook]& (V_8')^*\otimes F
\end{tikzcd}
\]
The maps in the bottom row are given by the identifications $\Hom(E_8,F)=V_8'$ and $\Hom(G_4,E_8)=(W_{10}')^*$. The composition of $E_8'^*\rightarrow\shE^*\rightarrow E_8$ gives a map $H^0(S',E_8')^*\rightarrow H^0(S,E_8)$ using the K\"unneth formula.
\begin{lemma}
    The map $H^0(S',E_8')^*\rightarrow H^0(S,E_8)$ above is an isomorphism.
\end{lemma}
\begin{proof}
    This can be proved in a similar way as in \cref{V_10 are dual to each other}.
\end{proof}
We use $Z'$ to denote the projective bundle $\bp_{S'}(E_8'^*)$ and use $\sho(-\eta)$ to denote the relative tautological subbundle.
\begin{proposition}
    There exists an embedding $Z'=\bp_{S'}(E_8'^*)$ into $\bp(H^0(S,E_8))$ using the line bundle $\sho(\eta)$. In particular, the vector bundle $E_8'$ is globally generated.
\end{proposition}
\begin{proof}
    We consider the map $\Hom(\sho,E)\otimes \Hom(E,E_8)\rightarrow \Hom(\sho,E_8)$, where $[E]=s'\in S'=M_{S}(3,h,5)$. This map can be obtained as follows: we consider the map $\sho(-\eta)\hookrightarrow E_8'^*\rightarrow \shE^* \rightarrow E_8$ over $Z'\times S$ and do the pushforward to the first factor to get the map $\sho(-\eta)\hookrightarrow E_8'^*\rightarrow H^0(S,E_8)$. The latter map $E_8'^*\rightarrow H^0(S,E_8)$ parametrizes the family of maps $\Hom(\sho,E)\otimes \Hom(E,E_8)\rightarrow \Hom(\sho,E_8)$ because $\Hom(\sho,E)=E_8'^*|_{s'}$ and $p_{2*}(\shE\otimes p_1^*E_8)=\sho_{S'}[0]$.

    Actually, the composition of $\sho(-\eta)\hookrightarrow E_8'^*\rightarrow H^0(S,E_8)$ exhibits $\sho(-\eta)$ as a subbundle of $H^0(S,E_8)$. To prove this, we only need to show that the composition of any nonzero map $\sho\rightarrow E$ with any nonzero map $E\rightarrow E_8$ remains nonzero. This comes from the fact that any nonzero map $E\rightarrow E_8$ is injective by the stability of $E$ and $E_8$.
    
    Consequently, we have an induced morphism from $Z'=\bp_{S'}(E_8'^*)$ to $\bp(H^0(S,E_8))$. The morphism is injective because if there exists a nonzero map $\sho\rightarrow E_8$ which could factor through two different $E_1,E_2$ corresponding to $s_1',s_2'\in S'$, then the first nonzero map in (\ref{K' ses}) would not be injective. 

    We consider the maps between tangent spaces induced by $Z'\rightarrow\bp(H^0(S,E_8))$. Let $z'=(s',\sigma')$ be a point of $Z'$, with $s'=[E]\in S'\cong M_{S}(3,h,5)$ and $\bc.\sigma'\subset H^0(S,E)=E'^*_{8,s'}$. A nonzero tangent vector $v_{s'}$ at $s'\in S'$ corresponds to a nontrivial extension 
\[\begin{tikzcd}
    0  \arrow{r}&E\arrow{r}\ar[dr,hook]& B\arrow{r}\ar[d,dashed] & E \arrow{r}& 0,\\
    &&E_8&&
\end{tikzcd}
\]
    where $B$ is equal to the Fourier-Mukai transform $(\phi_\shE)^{-1}[-2]$ of the structure sheaf of a length 2 subscheme centered at $s'\in S'$. A nonzero tangent vector at $z'\in Z'$, whose image via the $\bp^7$-bundle map equals $v_{s'}$, induces a nonzero map (via the Fourier-Mukai transform) $\sho_S=(\phi_\shE)^{-1}(E_8')[-2]\rightarrow B$ such that the composition $\sho_S\rightarrow B\rightarrow E$ is given by $\bc.\sigma'\subset H^0(S,E)$.
    Let $U_2$ be a two-dimensional subspace of $H^0(S,B)$ such that $U_2$ contains the image of $\bc.\sigma'$ and the image of $U_2$ in $H^0(S,E)$ is equal to $\bc.\sigma'$. Pick an extension $B\hookrightarrow E_8$ of the injection $E\hookrightarrow E_8$ as in \cref{maintheoremold}, which induces an injection $H^0(S,B)\hookrightarrow H^0(S,E_8)$. There exist different extensions $B\rightarrow E_8$ of the injection $E\hookrightarrow E_8$, but one can easily check that the images of $U_2$ under any induced maps $H^0(S,B)\rightarrow H^0(S,E_8)$ are the same and in particular are 2-dimensional. Hence we see that the tangent maps are injective and $Z'\rightarrow\bp(H^0(S,E_8))$ is an embedding.
\end{proof}

\subsection{Segre embedding using \(\sho_X(4H-h)\)}

Now we take global sections of $G_4\otimes W_{10}'^*\rightarrow E_8\rightarrow V_8'^*\otimes F$ and get $V_8\otimes W_{10}'^*\rightarrow H^0(E_8)\cong H^0(E_8')^*\rightarrow V_8'^*\otimes W_{10}$. The composition map $V_8\otimes W_{10}'^*\rightarrow V_8'^*\otimes W_{10}$ essentially arises from the map $V_8'\otimes W_{10}'^*\rightarrow V_{10}'\cong V_{10}^*\rightarrow V_8^*\otimes W_{10}$.

 Recall that we have maps $X'\hookrightarrow\mathds{P}(V_8)$ and $X'\hookrightarrow \mathds{P}(W_{10}'^*)$. Taking product and composing with the Segre morphism, we get an embedding of $X'$ into $\mathds{P}(V_8\otimes W_{10}'^*)$.  
 \begin{proposition}
          The linear span of $X'$ in $\mathds{P}(V_8\otimes W_{10}'^*)$ is contained in the projectivization of $\mathrm{Ker}(V_8\otimes W_{10}'^*\rightarrow H^0(E_8))$.
 \end{proposition}
 \begin{proof}
     We notice that $\mathrm{Ker}(V_8\otimes W_{10}'^*\rightarrow H^0(E_8))=\mathrm{Ker}(V_8\otimes W_{10}'^*\rightarrow V_8'^*\otimes W_{10})$, since $H^0(E_8)\rightarrow V_8'^*\otimes W_{10}$ is injective by \cref{E8 ses}. We would like to show that for any $x'=(s',f'^*)\in X'$ with associated $\bc.v\subset V_8$ and $f'^*\in F'^*|_{x'}\subset W_{10}'^*$,  the image of $v\otimes f'^*\in V_8\otimes W_{10}'^*$ in $V_8'^*\otimes W_{10}$ is zero when viewed as a map $[v\otimes f'^*]:V_8'\rightarrow W_{10}$. We unwind the definition and see that given any $v'\in V_8'$, we first take the pairing of it with $f'^*\in W_{10}'^*$ to get an element $v_{10}'\in V_{10}'$, and then take the pairing of $v_{10}'$ with $v_8$ to get an element $[v\otimes f'^*](v')\in W_{10}$ (using the map $V_8\otimes V_{10}'\cong V_8\otimes V_{10}^*\rightarrow W_{10}$). However we know that the image of the restriction of $V_8'\otimes \sho_{\mathds{P}(W_{10}'^*)}(-1)\rightarrow V_{10}'\otimes \sho_{\mathds{P}(W_{10}'^*)}$ to $X'\subset \bp(W_{10}'^*)$ is equal to $G_4'(-H')$ by \cite[Proposition 3.4.1]{Frederic}. Thus $v_{10}'\in G_4'(-H')|_{x'}\subset V_{10}'=V_{10}^*$, and it has zero pairing with $v\in V_8$ via the map $V_8\otimes V_{10}^*\rightarrow W_{10}$ by \cref{zero pairing}.  
 \end{proof}
 Notice that the pullback of the hyperplane bundle over $\mathds{P}(V_8\otimes W_{10}'^*)$ to $X'$ is $\sho_{X'}(4H'-h')$ and $\chi(\sho_{X'}(4H'-h'))=55$, $\chi(E_8)=25$. Also note that $\sho_{X'}(4H'-h')$ is the pullback of the Pl\"ucker polarization to $X'$ via the map $X'\rightarrow DV(t_2')\subset G(6,V_{10}'^*)$ constructed in \cite[Proposition 3.5.6]{Frederic}.  

Tensoring the last surjective map in (\ref{h-3H to V8' four terms}) with $\sho(h-H)$, we get a surjective map $E^*_8(h)\rightarrow \sho(4H-h)$.
\begin{Corollary}
    The surjective map $E^*_8(h)\rightarrow \sho(4H-h)$ induces an isomorphism on spaces of global sections.
\end{Corollary}
\begin{proof}
    All higher cohomology groups of $E_8^*(h)$ vanish by (\ref{E8 ses}). We only need to prove $H^1(X,\sho(2h-4H))=H^2(X,\sho(2h-4H))=0$. Using Serre duality, this reduces to showing $H^0(X,\sho(2H-h))=H^1(X,\sho(2H-h))=0$, which is equivalent to $H^0(S,(\mathrm{Sym}^2F)(-h))=H^1(S,(\mathrm{Sym}^2F)(-h))=0$. Since $F$ is a rigid simple bundle and $F\otimes F^*=F\otimes F(-h)=(\mathrm{Sym}^2F)(-h)\oplus \sho_S$, the claim follows.
\end{proof}

\subsection{Geometry of \(\mathds{P}_S(G_4)\)}

Denote $\mathds{P}_S(G_4)$ by $Y$ and denote the tautological hyperplane bundle by $\sho(\tau)$. By \cref{4 degenerate point}, we have the composition map $\mathcal{E}(-\tau)\rightarrow F'$ over $Y\times S'$, which fits into an exact sequence using Eagon-Northcott complex:
\begin{equation}
    0\rightarrow \sho(-3\tau-h)\rightarrow \mathcal{E}(-\tau)\rightarrow F'\rightarrow \mathscr{L}\rightarrow0,\label{2nd global Eagon-Northcott}
\end{equation}
where $\mathscr{L}$ is a torsion sheaf supported in the locus where the map $\mathcal{E}(-\tau)\rightarrow F'$ does not have full rank. We take the pushforward from $Y\times S'$ to $Y$ and get the following sequences on $Y$:
\begin{align}
    0\rightarrow \sho(-3\tau-h)\rightarrow E_8^*(-\tau)\rightarrow W_{10}'\rightarrow \mathscr{T}\rightarrow0\label{kernel H^0E to H^0F}\\
    0\rightarrow \mathscr{T}\rightarrow p_{Y*}\mathscr{L}\rightarrow\sho(-3\tau-h)\rightarrow 0
\end{align}
The sheaf $p_{Y*}\mathscr{L}$ is a rank 4 vector bundle by \cref{4 degenerate point}. As a result, $\mathscr{T}$ is a rank 3 vector bundle. These sequences globalize the sequences in \cite[Lemma 2.2.1]{Frederic}

Recall that we have an inclusion of vector bundles $G_4\otimes F^*\hookrightarrow V_{10}\otimes \sho_S$ over $S$ by (\ref{fundamentalSES}), which induces a morphism $\bp_S(G_4)\times_S\bp_S(F^*)\rightarrow \bp(V_{10})$. This morphism can be reinterpreted as follows. For $[E']=s\in S$, an element in $G_{4,s}$ corresponds to 
an element in $\Hom(E',F')$, and an element in $F^*_{s}$ corresponds to 
an element in $\Hom(G_4',E')$. The composition map $\Hom(G_4',E')\otimes \Hom(E',F')\rightarrow \Hom(G_4',F')=V_{10}'^*\cong V_{10}$ is exactly the fiber of the bundle map $G_4\otimes F^*\hookrightarrow V_{10}\otimes \sho_S$ over $s\in S$.

We consider the bundle map $V_{8}'\otimes\sho_{\bp(V_{10})}\rightarrow W_{10}'\otimes\sho_{\bp(V_{10})}(1)$ over $\bp(V_{10})$ induced by the morphism $V_8'\rightarrow W_{10}'\otimes V_{10}'= W_{10}'\otimes V_{10}^*$. When we pull back this bundle map to a point of $\bp_S(G_4)\times_S\bp_S(F^*)$, we exactly recover the map from $H^0(G_4')=V_8'$  to $H^0(F')=W_{10}'$ that is obtained using the composition map $\Hom(H^0(G_4'),H^0(E'))\otimes\Hom(H^0(E'),H^0(F'))\rightarrow \Hom(H^0(G_4'),H^0(F'))$.
\begin{proposition}
    The pullback of the bundle map $V_{8}'\otimes\sho_{\bp(V_{10})}\rightarrow W_{10}'\otimes\sho_{\bp(V_{10})}(1)$ to $\bp_S(G_4)\times_S\bp_S(F^*)$ is generically of rank 7, and the rank is at least 6 everywhere. The locus where the rank is 6 is a divisor of class $4H+2\tau-h$.\label{4H+2tau-h}
\end{proposition}
\begin{proof}
Any nonzero induced map $H^0(G_4')\rightarrow H^0(E')$ is of corank 1 by \cref{generic injective}, and any nonzero induced map $H^0(E')\rightarrow H^0(F')$ has a one-dimensional kernel by \cref{4 degenerate point}, which shows that the rank of the composition map $H^0(G_4')\rightarrow H^0(F')$ is at least 6 and at most 7. The rank 6 condition is equivalent to requiring that the composition of the map $\mathrm{Ker}(H^0(E')\rightarrow H^0(F'))\hookrightarrow H^0(E')$ with the quotient map $H^0(E')\rightarrow \mathrm{Coker}(H^0(G_4')\rightarrow H^0(E'))$ is zero. The map $\mathrm{Ker}(H^0(E')\rightarrow H^0(F'))\hookrightarrow H^0(E')$ is globally parametrized by $\sho(-2\tau-h)\hookrightarrow E_8^*$ (see (\ref{kernel H^0E to H^0F})) and the quotient map $H^0(E')\rightarrow \mathrm{Coker}(H^0(G_4')\rightarrow H^0(E'))$ is parametrized by $E^*_8\rightarrow \sho(4H-2h)$ (see (\ref{h-3H to V8' four terms})), which allow us to conclude.
\end{proof}
\begin{remark}
    The map $\sho(-2\tau-h)\rightarrow\sho(4H-2h)$ induces a map $\sho(h-2H)\rightarrow\sho(2H+2\tau)$ and thus a map $\sho(h-2H)\rightarrow S^2(G_4^*(H))$, which can actually be obtained from the embedding $X\hookrightarrow S'^{[2]}$: the kernel of the composition $\mathrm{Sym}^2(Q_4')\hookrightarrow \mathrm{Sym}^2(\mathrm{Sym}^2(F'^{[2]}))\rightarrow \mathrm{Sym}^4(F'^{[2]})$ is the line bundle $\sho_{S'^{[2]}}(h'-4\delta')$, such that the restriction of the subbundle    inclusion $\sho_{S'^{[2]}}(h'-4\delta')\hookrightarrow\mathrm{Sym}^2(Q_4')$ to $X$ is exactly the same map mentioned above. In particular, the map $\sho(h-2H)\rightarrow S^2(G_4^*(H))$ is of constant rank 1.
\end{remark}
These study allow us to directly show that the Debarre-Voisin type morphism $X\rightarrow G(6,V_{10}^*)$ induced by $t_2$ is an embedding.
\begin{Corollary}
    The bundle inclusion $G_4(-H)\hookrightarrow V_{10}\otimes\sho_X$ induces an embedding $X\hookrightarrow G(4,V_{10})$.
\end{Corollary}
\begin{proof}
    For each $x\in X$, we have associated to it the unique line $\bc.v'\in V_8'$ and the four-dimensional space $U_4:=G_4(-H)|_x\subset V_{10}\cong V_{10}'^*$. Because of \cref{4H+2tau-h}, \cref{zero pairing}, together with the fact that the induced map $\sho(h-2H)\rightarrow S^2(G_4^*(H))$ is of constant rank 1 (which in particular implies that the divisor in \cref{4H+2tau-h} cannot contain any $\mathbb{P}(U_4)$), for a general element of $U_4$, when viewed as a morphism $V_8'\rightarrow W_{10}'$, the kernel inside $V_8'$ is always the one-dimensional $\bc.v'$. By \cref{3H-h embedding}, the map sending $x\in X$ to $[\bc.v']\in \bp(V_8')$ is an embedding. These allow us to conclude that the map $X\rightarrow G(4,V_{10})$ is an embedding.
\end{proof}
Using the canonical isomorphism $V_{10}^*\cong V_{10}'$, we obtained a trivector $t_1\in \wedge^3V_{10}^*$ corresponding to $t_2'\in \wedge^3 V_{10}'$. Given such a trivector, one can construct the corresponding congruence of lines: $$Y_{t_1}^2=\{[V_2]\in G(2,V_{10})| \ t_1(a,b,c)=0,  \ \mathrm{for} \  \forall a,b\in V_2, \forall c\in V_{10}\}\subset G(2,V_{10}).$$
\begin{proposition}
    Via the inclusion $F^*(-\tau)\rightarrow G_4\otimes F^*\hookrightarrow V_{10}\otimes \sho_Y$, we get a map $Y\rightarrow G(2,V_{10})$, whose image lies in the congruence of lines $Y_{t_1}^2$. \label{t1peskine}
\end{proposition}
\begin{proof}
    Each element $\omega_1\in\wedge^2 V_8'$ induces a map $\omega_1: S^2W_{10}'^*\rightarrow\wedge^2 V_{10}'$ by the map $\wedge^2 V_8'\rightarrow \wedge^2 V_{10}'\otimes S^2W_{10}'$. By the construction of $\omega'\in\wedge^2 V_8'$ in \cite{Frederic}, the map $\omega': S^2W_{10}'^*\rightarrow\wedge^2 V_{10}'$ factors as $\omega': S^2W_{10}'^*\twoheadrightarrow V_{10}'^*\rightarrow\wedge^2 V_{10}'$ where the latter map in the factorization is given by $t_2'\in \wedge^3V_{10}'$. Hence we only need to prove that the composition of the map $\omega': S^2W_{10}'^*\rightarrow\wedge^2 V_{10}'\cong \wedge^2V_{10}^*$ with $\wedge^2 V_{10}^*\rightarrow \wedge^2(F(\tau))$ is zero.
    
    By \cref{omega 2 form}, the 2-form $\omega'$ lies in the kernel of $\wedge^2V_8'\rightarrow \wedge^2(E_8^*\otimes F)\rightarrow S^2(E_8^*)(h)$, and equivalently in the kernel of $\wedge^2V_8'\rightarrow \wedge^2[E_8^*(-\tau)\otimes F(\tau)]\rightarrow S^2(E_8^*(-\tau))\otimes \wedge^2(F(\tau))$. On the other hand, we have the natural map $E_8^*(-\tau)\hookrightarrow E_8^*\otimes G_4\rightarrow W_{10}'$, and then trivially $\omega'$ lies in the kernel of the composition of $\wedge^2V_8'\rightarrow S^2(E_8^*(-\tau))\otimes \wedge^2(F(\tau))\rightarrow S^2W_{10}'\otimes \wedge^2(F(\tau))$, thus giving rise to the zero map $S^2W_{10}'^*\rightarrow \wedge^2(F(\tau))$. (Essentially we are using \cref{interpretation} here).
\end{proof}

\printbibliography

@incollection {Mukai2,
    AUTHOR = {Mukai, Shigeru},
     TITLE = {Duality of polarized {$K3$} surfaces},
 BOOKTITLE = {New trends in algebraic geometry ({W}arwick, 1996)},
    SERIES = {London Math. Soc. Lecture Note Ser.},
    VOLUME = {264},
     PAGES = {311--326},
 PUBLISHER = {Cambridge Univ. Press, Cambridge},
      YEAR = {1999},
      ISBN = {0-521-64659-6},
   MRCLASS = {14J28 (14J60 18E30)},
  MRNUMBER = {1714828},
MRREVIEWER = {I.\ Dolgachev},
       DOI = {10.1017/CBO9780511721540.012},
       URL = {https://doi.org/10.1017/CBO9780511721540.012},
}

@article {Yoshioka,
    AUTHOR = {Yoshioka, K\B ota},
     TITLE = {Stability and the {F}ourier-{M}ukai transform. {II}},
   JOURNAL = {Compos. Math.},
  FJOURNAL = {Compositio Mathematica},
    VOLUME = {145},
      YEAR = {2009},
    NUMBER = {1},
     PAGES = {112--142},
      ISSN = {0010-437X,1570-5846},
   MRCLASS = {14F05 (14D20 14J28)},
  MRNUMBER = {2480497},
MRREVIEWER = {Carlo\ Giovanni\ Madonna},
       DOI = {10.1112/S0010437X08003758},
       URL = {https://doi.org/10.1112/S0010437X08003758},
}

@article {ToshiokaLemma,
    AUTHOR = {Yoshioka, K\B ota},
     TITLE = {Some examples of {M}ukai's reflections on {$K3$} surfaces},
   JOURNAL = {J. Reine Angew. Math.},
  FJOURNAL = {Journal f\"ur die Reine und Angewandte Mathematik. [Crelle's
              Journal]},
    VOLUME = {515},
      YEAR = {1999},
     PAGES = {97--123},
      ISSN = {0075-4102,1435-5345},
   MRCLASS = {14J28 (14J60)},
  MRNUMBER = {1717621},
MRREVIEWER = {I.\ Dolgachev},
       DOI = {10.1515/crll.1999.080},
       URL = {https://doi.org/10.1515/crll.1999.080},
}

@misc{YoshiokaExistence,
      title={Irreducibility of moduli spaces of vector bundles on K3 surfaces}, 
      author={Kota Yoshioka},
      year={2000},
      eprint={math/9907001},
      archivePrefix={arXiv},
      primaryClass={math.AG},
      url={https://arxiv.org/abs/math/9907001}, 
      howpublished = {{arXiv}:math/9907001 [math.{AG}]},
}

@incollection {FMofK3,
    AUTHOR = {Hosono, Shinobu and Lian, Bong H. and Oguiso, Keiji and Yau,
              Shing-Tung},
     TITLE = {Fourier-{M}ukai partners of a {$K3$} surface of {P}icard
              number one},
 BOOKTITLE = {Vector bundles and representation theory ({C}olumbia, {MO},
              2002)},
    SERIES = {Contemp. Math.},
    VOLUME = {322},
     PAGES = {43--55},
 PUBLISHER = {Amer. Math. Soc., Providence, RI},
      YEAR = {2003},
      ISBN = {0-8218-3264-6},
   MRCLASS = {14J28 (18E30)},
  MRNUMBER = {1987738},
MRREVIEWER = {Bal\'azs\ Szendr\H oi},
       DOI = {10.1090/conm/322/05678},
       URL = {https://doi.org/10.1090/conm/322/05678},
}

@article {YoshiokaStable,
    AUTHOR = {Yoshioka, K\B ota},
     TITLE = {Moduli spaces of stable sheaves on abelian surfaces},
   JOURNAL = {Math. Ann.},
  FJOURNAL = {Mathematische Annalen},
    VOLUME = {321},
      YEAR = {2001},
    NUMBER = {4},
     PAGES = {817--884},
      ISSN = {0025-5831,1432-1807},
   MRCLASS = {14D20 (14D22)},
  MRNUMBER = {1872531},
MRREVIEWER = {Jean-Marc\ Dr\'ezet},
       DOI = {10.1007/s002080100255},
       URL = {https://doi.org/10.1007/s002080100255},
}

@book {HuyBook,
    AUTHOR = {Huybrechts, Daniel and Lehn, Manfred},
     TITLE = {The geometry of moduli spaces of sheaves},
    SERIES = {Aspects of Mathematics},
    VOLUME = {E31},
 PUBLISHER = {Friedr. Vieweg \& Sohn, Braunschweig},
      YEAR = {1997},
     PAGES = {xiv+269},
      ISBN = {3-528-06907-4},
   MRCLASS = {14D20 (14F05)},
  MRNUMBER = {1450870},
MRREVIEWER = {Jean-Marc\ Dr\'ezet},
       DOI = {10.1007/978-3-663-11624-0},
       URL = {https://doi.org/10.1007/978-3-663-11624-0},
}

@misc{Frederic,
 author = {Han, Frédéric},
 title = {Geometry of genus sixteen {K3} surfaces},
 year = {2025},
 howpublished = {Preprint, {arXiv}:2501.16013 [math.{AG}]},
 url = {https://arxiv.org/abs/2501.16013},
 arXiv = {arXiv:2501.16013}
}

@article {BottVanishing,
    AUTHOR = {Totaro, Burt},
     TITLE = {Bott vanishing for algebraic surfaces},
   JOURNAL = {Trans. Amer. Math. Soc.},
  FJOURNAL = {Transactions of the American Mathematical Society},
    VOLUME = {373},
      YEAR = {2020},
    NUMBER = {5},
     PAGES = {3609--3626},
      ISSN = {0002-9947,1088-6850},
   MRCLASS = {14F17 (14J26 14J28)},
  MRNUMBER = {4082249},
MRREVIEWER = {Haohao\ Wang},
       DOI = {10.1090/tran/8045},
       URL = {https://doi.org/10.1090/tran/8045},
}

@article {DHOV,
    AUTHOR = {Debarre, Olivier and Han, Fr\'ed\'eric and O'Grady, Kieran and
              Voisin, Claire},
     TITLE = {Hilbert squares of {K}3 surfaces and {D}ebarre-{V}oisin
              varieties},
   JOURNAL = {J. \'Ec. polytech. Math.},
  FJOURNAL = {Journal de l'\'Ecole polytechnique. Math\'ematiques},
    VOLUME = {7},
      YEAR = {2020},
     PAGES = {653--710},
      ISSN = {2429-7100,2270-518X},
   MRCLASS = {14J42 (14C05 14J28)},
  MRNUMBER = {4086582},
MRREVIEWER = {Davide\ Cesare\ Veniani},
       DOI = {10.5802/jep.125},
       URL = {https://doi.org/10.5802/jep.125},
}

@article {Druel,
    AUTHOR = {Druel, St\'ephane},
     TITLE = {Quelques remarques sur la d\'ecomposition de {Z}ariski
              divisorielle sur les vari\'et\'es dont la premi\`ere classe de
              {C}hern est nulle},
   JOURNAL = {Math. Z.},
  FJOURNAL = {Mathematische Zeitschrift},
    VOLUME = {267},
      YEAR = {2011},
    NUMBER = {1-2},
     PAGES = {413--423},
      ISSN = {0025-5874,1432-1823},
   MRCLASS = {14C20 (14E30)},
  MRNUMBER = {2772258},
MRREVIEWER = {Zach\ Teitler},
       DOI = {10.1007/s00209-009-0626-4},
       URL = {https://doi.org/10.1007/s00209-009-0626-4},
}

@article {KapustkaVanGeemen,
    AUTHOR = {van Geemen, Bert and Kapustka, Grzegorz},
     TITLE = {Contractions of hyper-{K}\"ahler fourfolds and the {B}rauer
              group},
   JOURNAL = {Adv. Math.},
  FJOURNAL = {Advances in Mathematics},
    VOLUME = {412},
      YEAR = {2023},
     PAGES = {Paper No. 108814, 52},
      ISSN = {0001-8708,1090-2082},
   MRCLASS = {14J42 (14E05 14F22 14J60)},
  MRNUMBER = {4521697},
MRREVIEWER = {Alexandra\ Kuznetsova},
       DOI = {10.1016/j.aim.2022.108814},
       URL = {https://doi.org/10.1016/j.aim.2022.108814},
}

@misc{debarre2020hyperkahlermanifolds,
      title={Hyperk\"ahler manifolds}, 
      author={Olivier Debarre},
      year={2020},
      eprint={1810.02087},
      archivePrefix={arXiv},
      primaryClass={math.AG},
      url={https://arxiv.org/abs/1810.02087}, 
howpublished = {Preprint, {arXiv}:1810.02087 [math.{AG}] },
}

@incollection {Mukai16,
    AUTHOR = {Mukai, Shigeru},
     TITLE = {K3 surfaces of genus sixteen},
 BOOKTITLE = {Minimal models and extremal rays ({K}yoto, 2011)},
    SERIES = {Adv. Stud. Pure Math.},
    VOLUME = {70},
     PAGES = {379--396},
 PUBLISHER = {Math. Soc. Japan, [Tokyo]},
      YEAR = {2016},
      ISBN = {978-4-86497-036-5},
   MRCLASS = {14J28 (14M20)},
  MRNUMBER = {3618267},
MRREVIEWER = {Zhiyu\ Tian},
       DOI = {10.2969/aspm/07010379},
       URL = {https://doi.org/10.2969/aspm/07010379},
}

@article {Boucksom,
    AUTHOR = {Boucksom, S\'ebastien},
     TITLE = {Divisorial {Z}ariski decompositions on compact complex
              manifolds},
   JOURNAL = {Ann. Sci. \'Ecole Norm. Sup. (4)},
  FJOURNAL = {Annales Scientifiques de l'\'Ecole Normale Sup\'erieure.
              Quatri\`eme S\'erie},
    VOLUME = {37},
      YEAR = {2004},
    NUMBER = {1},
     PAGES = {45--76},
      ISSN = {0012-9593},
   MRCLASS = {32J18 (32C30)},
  MRNUMBER = {2050205},
MRREVIEWER = {Adam\ Gregory\ Harris},
       DOI = {10.1016/j.ansens.2003.04.002},
       URL = {https://doi.org/10.1016/j.ansens.2003.04.002},
}

@article {Ogrady,
    AUTHOR = {O'Grady, Kieran G.},
     TITLE = {Modular sheaves on hyperk\"ahler varieties},
   JOURNAL = {Algebr. Geom.},
  FJOURNAL = {Algebraic Geometry},
    VOLUME = {9},
      YEAR = {2022},
    NUMBER = {1},
     PAGES = {1--38},
      ISSN = {2313-1691,2214-2584},
   MRCLASS = {14J42 (14J60)},
  MRNUMBER = {4371547},
MRREVIEWER = {Filippo\ F.\ Favale},
       DOI = {10.14231/ag-2022-001},
       URL = {https://doi.org/10.14231/ag-2022-001},
}

@article {DebarreVoisin,
    AUTHOR = {Debarre, Olivier and Voisin, Claire},
     TITLE = {Hyper-{K}\"ahler fourfolds and {G}rassmann geometry},
   JOURNAL = {J. Reine Angew. Math.},
  FJOURNAL = {Journal f\"ur die Reine und Angewandte Mathematik. [Crelle's
              Journal]},
    VOLUME = {649},
      YEAR = {2010},
     PAGES = {63--87},
      ISSN = {0075-4102,1435-5345},
   MRCLASS = {14J35 (14C05 14J28 32J25 53C26)},
  MRNUMBER = {2746467},
MRREVIEWER = {Kieran\ G.\ O'Grady},
       DOI = {10.1515/CRELLE.2010.089},
       URL = {https://doi.org/10.1515/CRELLE.2010.089},
}

@article {oberdieck,
    AUTHOR = {Oberdieck, Georg},
     TITLE = {Gromov-{W}itten theory and {N}oether-{L}efschetz theory for
              holomorphic-symplectic varieties},
      NOTE = {With an appendix by Jieao Song},
   JOURNAL = {Forum Math. Sigma},
  FJOURNAL = {Forum of Mathematics. Sigma},
    VOLUME = {10},
      YEAR = {2022},
     PAGES = {Paper No. e21, 46},
      ISSN = {2050-5094},
   MRCLASS = {14N35 (14J28 14J42)},
  MRNUMBER = {4404412},
MRREVIEWER = {Yanqiao\ Ding},
       DOI = {10.1017/fms.2022.10},
       URL = {https://doi.org/10.1017/fms.2022.10},
}

@article {twistedcubic,
    AUTHOR = {Ellingsrud, Geir and Str\o mme, Stein Arild},
     TITLE = {The number of twisted cubic curves on the general quintic
              threefold},
   JOURNAL = {Math. Scand.},
  FJOURNAL = {Mathematica Scandinavica},
    VOLUME = {76},
      YEAR = {1995},
    NUMBER = {1},
     PAGES = {5--34},
      ISSN = {0025-5521,1903-1807},
   MRCLASS = {14N10 (14C05 14J30)},
  MRNUMBER = {1345086},
MRREVIEWER = {Susan\ J.\ Colley},
       DOI = {10.7146/math.scand.a-12522},
       URL = {https://doi.org/10.7146/math.scand.a-12522},
}

@article {KuznetsovG4,
    AUTHOR = {Kuznetsov, Alexander},
     TITLE = {Derived categories of families of {F}ano threefolds},
   JOURNAL = {Algebr. Geom.},
  FJOURNAL = {Algebraic Geometry},
    VOLUME = {12},
      YEAR = {2025},
    NUMBER = {4},
     PAGES = {519--574},
      ISSN = {2313-1691,2214-2584},
   MRCLASS = {14F08 (14J45 18G80)},
  MRNUMBER = {4927290},
}

@misc{HanCode,
  author       = {Han, Fr{\'e}d{\'e}ric},
  title        = {{Macaulay2} scripts accompanying
                  ``Geometry of genus sixteen {K3} surfaces''},
  howpublished = {\url{https://webusers.imj-prg.fr/~frederic.han/recherche/documents/genre16-m2/}},
  note         = {In particular, genre16-Xfunctionsonly.m2,
                  function beta. Accessed 17 October 2026}
}
\texttt{Junyu.Meng@math.univ-toulouse.fr} \\
Institut de Math\'ematiques de Toulouse; UMR 5219, Universit\'e de Toulouse; CNRS, UPS, F-31062 Toulouse Cedex 9, France

\end{document}